\documentclass[a4paper,11pt]{amsart}

\usepackage[utf8x]{inputenc}
\usepackage[english]{babel}
\usepackage{geometry}
\usepackage{fullpage}
\usepackage{hyperref}
\usepackage{amsmath}
\usepackage{amsthm}
\usepackage{amsfonts}
\usepackage{amssymb}
\usepackage{amsxtra}
\usepackage{graphicx}
\usepackage{epsfig}
\usepackage{enumerate}
\usepackage{mathrsfs}
\usepackage{xfrac}
\usepackage{tikz}
\usetikzlibrary{arrows}
\usepackage{latexsym}
\usepackage{titletoc}
\usepackage{indentfirst}
\usepackage{stmaryrd}

\swapnumbers
\theoremstyle{plain}
\newtheorem{prop}{Proposition}[section]
\newtheorem{theorem}[prop]{Theorem}
\newtheorem{theoremintro}{Theorem}
\newtheorem{lemma}[prop]{Lemma}
\newtheorem{cor}[prop]{Corollary}
\theoremstyle{definition}
\newtheorem{definit}[prop]{Definition}
\newtheorem{parag}[prop]{}
\theoremstyle{remark}
\newtheorem{remark}[prop]{\textsc{Remark}}
\newtheorem{notat}[prop]{\textsc{Notation}}

\newcommand{\Art}{\text{Art}}

\newcommand{\bbP}{\mathbb{P}}
\newcommand{\bbQ}{\mathbb{Q}}

\newcommand{\bbZ}{\mathbb{Z}}

\newcommand{\bark}{\overline{k}}

\newcommand{\barphi}{\overline{\phi}}
\newcommand{\bartheta}{\overline{\theta}}

\newcommand{\barx}{\overline{x}}

\newcommand{\coh}{\text{coh}}

\newcommand{\age}{\text{age}}
\newcommand{\cl}{\text{cl}}

\newcommand{\et}{\text{\'et}}

\newcommand{\Hom}{\text{Hom}}

\newcommand{\id}{\text{id}}
\newcommand{\ob}{\text{ob}}
\newcommand{\gl}{\text{gl}}
\newcommand{\insieme}[2]{\left\{\,#1\,\middle|\, #2\,\right\}}

\newcommand{\st}{\text{st}}
\newcommand{\ua}{\underline{a}}

\newcommand{\calC}{\mathcal{C}}

\newcommand{\calM}{\mathcal{M}}
\newcommand{\calN}{\mathcal{N}}
\newcommand{\calR}{\mathcal{R}}

\newcommand{\calX}{\mathcal{X}}
\newcommand{\calY}{\mathcal{Y}}

\newcommand{\frakE}{\mathfrak{E}}

\newcommand{\frakm}{\mathfrak{m}}

\newcommand{\scrC}{\mathscr{C}}

\newcommand{\rank}{\mathfrak{r}}
\newcommand{\scrL}{\mathscr{L}}
\newcommand{\scrO}{\mathscr{O}}

\newcommand{\Sch}{\text{Sch}}

\newcommand{\virt}{\text{virt}}

\newcommand{\I}{\text{I}}
\newcommand{\inertia}{\overline{\mathcal{I}}_{\mu}(\mathcal{X})}
\newcommand{\Ignb}{\text{I}_{g,n,\beta}^{\mathcal{X}}}
\newcommand{\Ignbe}{\text{I}_{g,n,\beta_\eta}^{\mathcal{X}}}
\newcommand{\Inb}{\text{I}_{0,n,\beta}^{\mathcal{X}}}

\newcommand{\KgnXd}{{\mathcal{K}}_{g,n}(\sfrac{\mathcal{X}}{S}, d)}
\newcommand{\KgnX}{{\mathcal{K}}_{g,n}(\mathcal{X}, \beta)}
\newcommand{\KgnXb}{{\mathcal{K}}_{g,n}(\sfrac{\mathcal{X}}{S}, \beta)}
\newcommand{\KgnXbk}{{\mathcal{K}}_{g,n}(\sfrac{\mathcal{X}}{k}, \beta)}

\newcommand{\KgnXbe}{{\mathcal{K}}_{g,n}(\sfrac{\calX}{S}, \beta_\eta)}

\newcommand{\MgnStw}{\mathfrak{M}_{\sfrac{g,n}{S}}^{\text{tw}}}
\newcommand{\Mgnktw}{\mathfrak{M}_{\sfrac{g,n}{k}}^{\text{tw}}}
\newcommand{\Mgntw}{\mathfrak{M}_{g,n}^{\text{tw}}}
\newcommand{\mgnS}{\overline{\mathcal{M}}_{\sfrac{g,n}{S}}}

\newcommand{\mgnk}{\overline{\mathcal{M}}_{\sfrac{g,n}{k}}}
\renewcommand{\phi}{\varphi}

\DeclareMathOperator{\aut}{Aut}

\DeclareMathOperator{\rk}{rk}

\DeclareMathOperator{\spec}{Spec}

\author{Flavia Poma}
\title{Gromov-Witten theory of tame Deligne-Mumford stacks in mixed characteristic}
\date{\today\\
MSC classes: 14B35, 14H10, 14A20}

\begin{document}
\begin{abstract}
  We define Gromov-Witten classes and invariants of smooth proper tame
  Deligne-Mumford stacks of finite presentation over a Dedekind
  domain. We prove that they are deformation invariants and verify the
  fundamental axioms. For a smooth proper tame Deligne-Mumford stack
  over a Dedekind domain, we prove that the invariants of fibers in
  different characteristics are the same. We show that genus zero
  Gromov-Witten invariants define a potential which satisfies the WDVV
  equation and we deduce from this a reconstruction theorem for genus
  zero Gromov-Witten invariants in arbitrary characteristic.
\end{abstract}
\maketitle
\setcounter{tocdepth}{1}
\tableofcontents
\section{Introduction}
Gromov-Witten theory of orbifolds was introduced in the symplectic
setting in \cite{CR} and in the algebraic setting in \cite{AGV02} and
\cite{AGV08}, where Abramovich, Graber and Vistoli developed the
Gromov-Witten theory of Deligne-Mumford stacks in characteristic zero,
using the moduli stack of twisted stable maps into $\mathcal{X}$,
denoted by $\mathcal{K}_{g,n} (\mathcal{X}, \beta)$. This stack was
constructed in \cite{AV02} and it is the necessary analogue of
Kontsevich’s moduli stack of stable maps for smooth projective
varieties when replacing the variety with a Deligne-Mumford stack.

In this paper we define Gromov-Witten classes and invariants
associated to smooth proper tame Deligne-Mumford stacks of finite presentation over a
Dedekind domain. The main motivation for us is to compare the
invariants in different characteristics for stacks defined in mixed
characteristic. We hope that this approach could give a useful insight
into the Gromov-Witten theory in characteristic zero, providing a new
technique for computing Gromov-Witten invariants.
% (\textbf{TODO:  espandere?}).

We consider a modified version, which we denote by $\KgnXbe$, of
Abramovich, Graber and Vistoli' stack of twisted stable maps. The
stack $\KgnXbe$ parametrizes twisted stable maps to $\calX$, but we
take $\beta_\eta$ to be a cycle class over the generic fiber
$\calX_\eta$ of $\calX$ rather than over $\calX$ itself
(section~\ref{section_stable}). This stack turns out to be more convenient
when we want to compare the Gromov-Witten invariants in mixed
characteristic.

The fundamental ingredient for the construction of Gromov-Witten
invariants is the virtual fundamental class ${[\mathcal{K}_{g,n}
  (\mathcal{X}, \beta)]}^{\virt}\in A_*(\mathcal{K}_{g,n}
(\mathcal{X}, \beta))$. In the language of \cite{BF}, a virtual
fundamental class ${[\mathcal{M}]}^\virt\in A_*(\mathcal{M})$ is
defined in the Chow group with rational coefficients, for a
Deligne-Mumford stack $\mathcal{M}$ endowed with a perfect obstruction
theory. The main problem in developing Gromov-Witten theory in positive or
mixed characteristic is that in general the stack $\KgnX$ is not
Deligne-Mumford. For istance this happens for
$\mathcal{K}_{0,0}(\bbP_k^1,p)$, when $k$ is a field of characteristic
$p > 0$, because the map $f \colon \bbP_k^1 \rightarrow \bbP_k^1$ such
that $f (x_0 , x_1) = (x_0^p, x_1^p)$ is stable but has
stabilizer \[\mu_p = \spec \sfrac{k[x]}{(x^p − 1)} = \spec
\sfrac{k[x]}{{(x − 1)}^p}\text{,}\] which is not reduced. When the
base is a field of characteristic $p>0$, then $\KgnX$ is still
Deligne-Mumford for certain values of the fixed discrete parameters
$g,n,\beta$ which are big with respect to $p$ (\cite{AGV02}). However,
this is not satisfactory from the point of view of Gromov-Witten
theory, because most of the properties of Gromov-Witten invariants
(e.g. WDVV equation, Getzler relations) involve all the invariants at
the same time.

The definition of virtual fundamental class for Artin stacks was not
feasible at the beginning because of the lack for Artin stacks of two
useful technical devices: Chow groups and the cotangent complex. We
now have these devices at our disposal.  Chow groups and intersection
theory for Artin stacks over a field are defined in \cite{kresch}. A
working theory for the cotangent complex of a morphism of Artin stacks
is provided by \cite{laumon}, \cite{olsson}, \cite{LO}. Nonetheless
the presence of these tools is not enough to overcome all the
difficulties in the absolute case. However, for the purpose of this
work, it is enough to define a relative version of the virtual
fundamental class of an Artin stack. The crucial point is to observe
that both Kresch's intersection theory and the construction of the
virtual fundamental class in \cite{BF} 7 generalize to Artin stacks
over a Dedekind domain. In section~\ref{section_virtualclass} we apply
this to the natural forgetful functor $\theta \colon \KgnX \rightarrow
\Mgntw$ into the stack of twisted curves $\Mgntw$ constructed in
\cite{AV02}, after we exhibited a perfect relative obstruction theory
for $\theta$, and we construct a virtual fundamental class
${[\KgnX]}^{\virt}\in A_*(\KgnX)$.

A Dedekind domain $D$ can be thought of as a space whose points
corresponds to fields of different characteristics; a Deligne-Mumford
stack $\mathcal{Y}$ over $D$ is a family of Deligne-Mumford stacks -
the fibers - each of which is defined over a point of $D$. We prove
the following result, provinding a comparison between invariants in
different characteristics (section~\ref{section_GW}).
\begin{theoremintro}
  Let $\mathcal{Y}$ be a smooth proper tame Deligne-Mumford stack of finite
  presentation over a Dedekind domain $D$. Then the Gromov-Witten
  theories of the geometric fibers of $\mathcal{Y}$ are equivalent (i.e., the
  Gromov-Witten invariants of the fibers are the same).
\end{theoremintro}
When the base is an algebraically closed field $k$, we
prove that Gromov-Witten invariants define an associative and
supercommutative product on the quantum cohomology ring
\[H_\st^*(\mathcal{X})= \sum_r H^r(\inertia,
\bbQ_l(\overline{r}))\text{,}\] where the right hand side is the
$l$-adic \'etale cohomology, for a prime $l$ different from the
characteristic of $k$, of the rigidified
ciclotomic inertia stack $\inertia$ (section~\ref{section_genus0}).
\subsection*{Future work}
A natural generalization would be to develop a Gromov-Witten theory
for tame Artin stacks, using the moduli stack of twisted stable maps
constructed in \cite{AOV11}. The main problem is that the natural
forgetful functor $\theta \colon \KgnX \rightarrow \Mgntw$ is not of
Deligne-Mumford type in general, and therefore the relative cotangent
complex of $\theta$ has three terms, so that one cannot use the
construction described in \ref{section_virtualclass}.

In another direction, it would be interesting to prove a degeneration
formula in the mixed characteristic setting. This would give a useful
tool to compute Gromov-Witten invariants of Deligne-Mumford stacks in
characteristic zero out of \emph{simpler} invariants of tame
Deligne-Mumford stacks in positive characteristic. For istance, this
would apply to the fake projective plane constructed by Mumford in
\cite{mumford} using $p$-adic uniformization. We imagine this is far
from easy, but we hope to return to these points in a future paper.
\subsection*{Acknowledgements}
I am grateful to my advisor, Angelo Vistoli for his support and
valuable conversations and suggestions. I would like to thank my
internal advisor, Barbara Fantechi, for introducing me to the problem
presented here and for helpful discussions.  I learned a lot during a
three mounths stay at Stanford University. I would like to thank
Prof. Jun Li and Prof. Ravi Vakil for their hospitality. Grateful
thanks are extended for the wonderful work enviroment. I am indebted
to Scuola Normale Superiore and Dipartimento di Matematica of
University of Pisa for their hospitality during a eight months stay;
part of this work (including the generalization of results to
Deligne-Mumford stacks) was done during this period. I would like to
thank Prof. Yuri Manin and Prof. Richard Thomas for valuable comments
and suggestions of further applications of the results presented here.
A special thank is due to Stefano Maggiolo and Fabio Tonini for
providing
{\href{http://people.sissa.it/~maggiolo/software/commu.php}{{\tt
      Commu}}}, a software to draw commutative diagrams in \LaTeX,
which has been of great help during the drafting of this paper.  I was
supported by \textsc{sissa}, \textsc{in}d\textsc{am} and
\textsc{gnsaga}, \textsc{miur}.
\subsection*{Notations} We write $(\sfrac{\Sch}{S})$ for the category
of schemes over a base scheme $S$. For a scheme $X\in
(\sfrac{\Sch}{S})$, we denote by $A_*(\sfrac{X}{S})$ the group of
numerical equivalence classes of cycles. All stacks are Artin stacks
in the sense of \cite{artin}, \cite{laumon} and are of finite type
over a base scheme. Unless otherwise specified, the words "stack of
twisted stable maps" refer to $\KgnXbe$ in
Definition~\ref{def_stmaps}. We recall that a Deligne-Mumford stack
$\calX$ is tame if for every morphism $\barx \colon \spec k
\rightarrow \calX$, with $k$ algebraically closed, the stabilizer
group of $\barx$ in $\calX$ has order invertible in $k$.
\section{The stack of twisted stable maps}\label{section_stable}
Let $D$ be a Dedekind domain and set $S=\spec D$. Let $\calX$ be a
proper tame Deligne-Mumford stack of finite presentation over $S$,
admitting a projective coarse moduli scheme $X$. We fix an ample
invertible sheaf $\scrO(1)$ on $X$. We fix integers $g\geq 0$, $n\geq
0$. Let $\eta$ be the generic point of $S$ and set $X_\eta= X \times_S
\eta$. Fix $\beta_\eta\in A_1(\sfrac{X_\eta}{\eta})$.
\subsection{Twisted curves and twisted stable maps}
For any closed point $s\in S$, we denote by $X_s$ the fiber over
$s$. Let $\frakm_s\subset D$ be the maximal ideal corresponding to $s$
and consider the localization $R=D_{\frakm_s}$ of $D$ at
$\frakm_s$. Let us set $\widetilde{X}_s=X \times_S\spec R$ and let
$X_s \xrightarrow{i} X$ and $X_\eta \xrightarrow{j} X$ be the natural
inclusions. Notice that $R$ is a dicrete valuation ring and, by
\cite{fulton} 20.3, there exists a specialization
homomorphism \[\sigma_s \colon A_*(\sfrac{X_\eta}{\eta})\rightarrow
A_*(\sfrac{X_s}{s})\text{,}\] sending a cycle $\alpha$ to
$i^!\widetilde{\alpha}$, for some $\widetilde{\alpha}\in
A_*(\sfrac{\widetilde{X}_s}{R})$ such that
$j^*\widetilde{\alpha}=\alpha$. By \cite{fulton} 20.3.5, there exists
an induced specialization homomorphism \[\overline{\sigma_s} \colon
A_*(\sfrac{X_{\overline{\eta}}}{\overline{\eta}})\rightarrow
A_*(\sfrac{X_{\overline{s}}}{\overline{s}})\text{,}\] where
$\overline{\eta}$ and $\overline{s}$ are geometric points over $\eta$
and $s$. We denote by $\overline{\beta_\eta}\in
A_1(\sfrac{X_{\overline{\eta}}}{\overline{\eta}})$ the cycle class
induced by $\beta_\eta$ and we notice that
$\overline{\sigma_s}\left(\overline{\beta_\eta}\right)=
\overline{\sigma_s(\beta_\eta)}$.
\begin{definit}
  Let $T$ be a scheme over $S$. A \emph{stable $n$-pointed map of genus $g$ and class $\beta_\eta$
    into $X$} is the data $(C\xrightarrow{\pi} T, t_i, f)$, where
\begin{enumerate}
\item the morphism $\pi$ is a projective flat family of curves;
\item the geometric fibers of $\pi$ are reduced with at most nodes as
  singularities;
\item the sheaf $\pi_*\omega_{\sfrac{C}{T}}$ is locally free of rank g
  (where $\omega_{\sfrac{C}{T}}$ is the relative dualizing sheaf);
\item the morphisms $t_1, \ldots, t_n$ are sections of $\pi$ which are
  disjoint and land in the smooth locus of $\pi$;
\item $f \colon C \rightarrow X$ is a morphism of $S$-schemes;
\item the group scheme $\aut(C,f,\pi,t_i)$ of automorphisms of $C$,
  which commute with $f$, $\pi$ and $t_i$, is finite over $T$;
\item for every geometric point $\overline{t}\rightarrow T$, we
  consider the following induced morphisms \[C_{\overline{t}}=C
  \times_T \overline{t} \xrightarrow{f_{\overline{t}}}
  X_{\overline{t}}=X \times_S \overline{t}\xrightarrow{\tau}
  X_{\overline{s}}=X \times_S \overline{s}\rightarrow X_s=X \times_S
  s\xrightarrow{i}X\text{,}\] where $s=\spec k\in S$ is the image of
  $\overline{t}$ and $\overline{s}=\spec \bark$, with $\bark$ a
  separable closure of $k$, then we have
  ${f_{\overline{t}}}_*[C_{\overline{t}}]=
  \tau^*\overline{\sigma_s}\left(\overline{\beta_\eta}\right)$.
\end{enumerate}
\end{definit}
\begin{remark}
  Notice that a stable map of class $\beta_{\eta}$ is a stable map of
  class $\beta$ (in the sense of \cite{AV02} 4.3.1) for some $\beta\in
  A_1(\sfrac{X}{S})$ such that $j^*\beta=\beta_\eta$.
\end{remark}
\begin{definit}
  Let $T$ be a scheme over $S$. A \emph{twisted stable $n$-pointed map of genus $g$ and class
    $\beta_\eta$ into $\calX$} over $T$ is the data $(\calC \rightarrow T,
  {\{\Sigma_i^\calC\}}_{i=1}^n, f \colon \calC \rightarrow \calX)$
 where
\begin{enumerate}
\item the following natural diagram is commutative
  \[
  \begin{tikzpicture}[xscale=1.5,yscale=-1.2]
    \node (A0_0) at (0, 0) {$\calC$};
    \node (A0_1) at (1, 0) {$\calX$};
    \node (A1_0) at (0, 1) {$C$};
    \node (A1_1) at (1, 1) {$X$};
    \path (A0_0) edge [->]node [auto] {$\scriptstyle{f}$} (A0_1);
    \path (A1_0) edge [->]node [auto] {$\scriptstyle{f}$} (A1_1);
    \path (A0_1) edge [->]node [auto] {$\scriptstyle{}$} (A1_1);
    \path (A0_0) edge [->]node [auto] {$\scriptstyle{}$} (A1_0);
  \end{tikzpicture}
  \]
\item $(\calC \rightarrow T, {\{\Sigma_i^\calC\}}_{i=1}^n)$ is a twisted
  nodal $n$-pointed curve of genus $g$ over $T$;
\item the morphism $\calC \rightarrow \calX$ is representable;
\item let $\Sigma_i^C$ be the image of $\Sigma_i^\calC$ in $C$, then
  $(C → T, {\{\Sigma_i^C\}}_{i=1}^n, f \colon C \rightarrow X)$ is a
  stable $n$-pointed map of class $\beta_\eta$.
\end{enumerate}
\end{definit}
\begin{definit}\label{def_stmaps}
  We denote by $\KgnXbe$ the category fibered in groupoids over
  $(\sfrac{\Sch}{S})$ of twisted stable $n$-pointed maps of genus $g$ and
  class $\beta_\eta$ into $\calX$.
\end{definit}
\begin{theorem}\label{theorem_kbe}
  The category $\KgnXbe$ is a proper Artin stack over $S$, admitting a
  projective coarse moduli scheme
  ${K}_{g,n}(\sfrac{\calX}{S},\beta_\eta)\rightarrow S$.
\end{theorem}
\begin{proof}
  Let $d=\deg \beta_\eta$. It is enough to show that $\KgnXbe$ is an
  open and closed substack of $\KgnXd$ and then apply \cite{AV02}
  1.4.1. Notice that $\KgnXbe = \bigsqcup \KgnXb$, where the union is
  over $\beta \in A_1(\sfrac{X}{S})$ such that $j^*\beta =
  \beta_\eta$. By \cite{AV02} 1.4.1, $\KgnXbe$ is an open substack of
  $\KgnXd$, because it is a union of open substacks. On the other hand
  $\KgnXd \setminus \KgnXbe = \bigsqcup \KgnXb$ is open, where the
  union is over $\beta \in A_1(\sfrac{X}{S})$ such that $\deg
  \beta=d$, $j^*\beta \neq \beta_\eta$. It follows that $\KgnXbe$ is a
  closed substack of $\KgnXd$.
\end{proof}
\begin{parag}
We denote by $\MgnStw$ the stack of twisted $n$-pointed curves of
genus $g$ as defined in \cite{AV02} 4.1.2. Recall that $\MgnStw$ is a
smooth Artin stack, locally of finite type over $S$. Moreover, the
stack $\mathfrak{M}_{\sfrac{g,n}{S}}^{\text{tw}, \leq N, \Gamma}$,
classifying twisted curves $(\calC, \{\Sigma_i^\calC\})$ such that the
order of the stabilizer group at every point is at most $N$ and the
coarse space $C$ of $\calC$ has dual graph $\Gamma$, is a smooth Artin
stack of finite type over $S$ (\cite{O02} 1.9--1.12).
\end{parag}
\begin{definit}\label{def_morphisms}
  Let $\calC\rightarrow \MgnStw$ be the universal twisted nodal curve. We define the algebraic
  stack $\underline{\Hom}_{\Mgntw}(\calC,\calX)$ over $\MgnStw$ as follows
\begin{enumerate}
\item for every $S$-scheme $T$, an object of $\underline{\Hom}_{\Mgntw}(\calC, \calX)(T)$ is a
  twisted pointed curve $(\calC_T \rightarrow T, {\{\Sigma_i\}}_{i=1}^n)$ over $T$
  together with a representable morphism of $S$-stacks $f \colon \calC_T \rightarrow
  \calX$;
\item a morphism from $(\calC_T \rightarrow T, {\{\Sigma_i^\calC\}},
  f)$ to $(\calC_{T'} \rightarrow T', {\{\Sigma_i'\}}, f')$ consists
  of data $(F, \alpha)$, where $F \colon \calC_T \rightarrow
  \calC_{T'}$ is a morphism of twisted curves and $\alpha
  \colon f \rightarrow f' \circ F$ is an isomorphism.
\end{enumerate}
\end{definit}
\begin{remark}
There is a canonical functor $\bartheta \colon \underline{\Hom}_{\Mgntw}(\calC, \calX)
\rightarrow \MgnStw$ which forgets the map into $\calX$. Moreover, since
stability is an open condition, the stack $\KgnXbe$ is an open
substack of $\underline{\Hom}_{\Mgntw}(\calC, \calX)$.
\end{remark}
\begin{prop}
The natural forgetful functor \[\theta \colon \KgnXbe \rightarrow
\MgnStw\] which forgets the morphism into $\calX$ is of Deligne-Mumford type. 
\end{prop}
\begin{proof}
  Let $U\rightarrow \MgnStw$ be a morphism from a scheme $U$ over $S$
  and let us denote $\calC_U=\calC \times_{\Mgntw}U$ the corresponding
  twisted pointed curve over $U$. Form the fiber diagram
  \[
  \begin{tikzpicture}[xscale=2.2,yscale=-1.6]
    \node (A0_0) at (0, 0) {$V$};
    \node (A0_1) at (1, 0) {$\underline{\Hom}_{\Mgntw}(\calC, \calX)$};
    \node (A1_0) at (0, 1) {$U$};
    \node (A1_1) at (1, 1) {$\MgnStw$};
    \path (A0_0) edge [->]node [auto] {$\scriptstyle{}$} (A0_1);
    \path (A1_0) edge [->]node [auto] {$\scriptstyle{}$} (A1_1);
    \path (A0_1) edge [->]node [auto] {$\scriptstyle{\bartheta}$} (A1_1);
    \path (A0_0) edge [->]node [auto] {$\scriptstyle{}$} (A1_0);
  \end{tikzpicture}
  \]
  and notice that $V=\underline{\Hom}_{U}(\calC_U, \calX)$. Since
  $\calC_U$ and $\calX$ are Deligne-Mumford stacks it follows, by
  \cite{Ol} 1.1, that $V$ is a Deligne-Mumford stack and hence
  $\bartheta$ is of Deligne-Mumford type. The statement follows from
  the fact that $\KgnXbe$ is an open substack of
  $\underline{\Hom}_{\Mgntw}(\calC, \calX)$.
\end{proof}
\begin{remark}\label{remark_univ}
For every $S$-scheme $T$, a morphism $T \rightarrow
\KgnXbe$ corresponds to a stable map $(C_T\xrightarrow{\pi_T} T, t_i,
f_T)$ over $T$, then, by descent theory, the identity of $\KgnXbe$
corresponds to a universal stable map $(\scrC \xrightarrow{\pi}
\KgnXbe,\sigma_i, \psi)$.
\end{remark}
\subsection{Evaluation maps}
Let $\overline{\mathcal{I}}_{\mu}(\calX)$ be stack of cyclotomic
gerbes of $\calX$ as described in \cite{AGV08} 3.3. Recall that
$\overline{\mathcal{I}}_{\mu}(\calX)$ is proper, since $\calX$ is
proper; moreover, if $\calX$ is smooth then
$\overline{\mathcal{I}}_{\mu}(\calX)$ is smooth (\cite{AGV08} 3.4).
\begin{remark}[\cite{AGV08} 3.5]\label{remark_inv}
There is an involution $\iota \colon
\overline{\mathcal{I}}_{\mu}(\calX) \rightarrow
\overline{\mathcal{I}}_{\mu}(\calX)$ defined over each
$\overline{\mathcal{I}}_{\mu_r}(\calX)$ as follows. Consider the
inversion automorphism $\tau \colon \mu_r \rightarrow \mu_r$ sending
$\xi$ to $\xi^{-1}$. For every object $(\mathcal{G}, \psi)$ of
$\overline{\mathcal{I}}_{\mu_r}(\calX)$, we can change the banding of
the gerbe $\mathcal{G}\rightarrow T$ through $\tau$ and get another
object $ ^\tau\mathcal{G}\rightarrow \calX$ of
$\overline{\mathcal{I}}_{\mu_r}(\calX)$.
\end{remark}
\begin{notat}\label{delta}
  We denote $\Delta \colon \inertia \rightarrow \inertia^2$ the
  morphism, which we will call \emph{diagonal}, induced by the
  identity and the involution $\iota$.
\end{notat}
\begin{definit}[\cite{AGV08} 4.4.1]
  The \emph{$i$-th evaluation map} $e_i\colon \KgnXbe \rightarrow
  \overline{\mathcal{I}}_{\mu}(\calX)$ is the morphism that associates
  to every twisted stable map $(\calC \rightarrow T,
  {\{\Sigma_i^\calC\}}_i, f \colon \calC \rightarrow \calX)$ the
  diagram
  \[
  \begin{tikzpicture}[xscale=1.5,yscale=-1.2]
    \node (A0_0) at (0, 0) {$\Sigma_i^\calC$};
    \node (A0_1) at (1, 0) {$\calX$};
    \node (A1_0) at (0, 1) {$T$};
    \path (A0_0) edge [->]node [auto] {$\scriptstyle{f}$} (A0_1);
    \path (A0_0) edge [->]node [auto] {$\scriptstyle{}$} (A1_0);
  \end{tikzpicture}
  \]
  The \emph{$i$-th twisted evaluation map} $\check{e}_i\colon
  \KgnXbe \rightarrow \overline{\mathcal{I}}_{\mu}(\calX)$ is the
  composition $\iota \circ e_i$, where $\iota$ is the involution
  described in Remark~\ref{remark_inv}.
\end{definit}
\begin{remark}
  Let us notice that the evaluation map $e_i$ is the
  composition \[\KgnXbe \xrightarrow{\Gamma_{e_i}} \KgnXbe
  \times_S\overline{\mathcal{I}}_{\mu}(\calX)\xrightarrow{\pi}
  \overline{\mathcal{I}}_{\mu}(\calX)\text{,}\] where $\Gamma_{e_i}$
  is the graph of $e_i$ and $\pi$ is the projection. By the following
  cartesian diagram
  \[
  \begin{tikzpicture}[xscale=4.4,yscale=-2.4]
    \node (A0_0) at (0, 0) {$\KgnXbe$};
    \node (A0_1) at (1, 0) {$\overline{\mathcal{I}}_{\mu}(\calX)$};
    \node (A1_0) at (0, 1) {$\KgnXbe \times_S \overline{\mathcal{I}}_{\mu}(\calX)$};
    \node (A1_1) at (1, 1) {$\overline{\mathcal{I}}_{\mu}(\calX) \times_S \overline{\mathcal{I}}_{\mu}(\calX)$};
    \path (A0_0) edge [->]node [auto] {$\scriptstyle{e_i}$} (A0_1);
    \path (A0_0) edge [->]node [left] {$\scriptstyle{\Gamma_{e_i}}$} (A1_0);
    \path (A0_1) edge [->]node [auto] {$\scriptstyle{\Delta}$} (A1_1);
    \path (A1_0) edge [->]node [auto] {$\scriptstyle{e_i \times \id}$} (A1_1);
  \end{tikzpicture}
  \]
  it follows that $\Gamma_{e_i}$ is a regular local immersion, hence,
  by \cite{kresch} 6.1, there exists a Gysin map
  $\Gamma_{e_i}^!$. Moreover $\underline{\Hom}_{\Mgntw}(\calC,\calX)$
  is flat over $\Mgntw$ therefore, since $\Mgntw$ is smooth over $S$
  and $\KgnXbe$ is an open substack of
  $\underline{\Hom}_{\Mgntw}(\calC,\calX)$, we get that $\pi$ is
  flat. Then we can define the pull-back $e_i^*=\Gamma_{e_i}^!\circ
  \pi^*$.
\end{remark}
\begin{notat}
  We write $e^*(\underline{\gamma})=\bigsqcap_{i=1}^ne_i^*(\gamma_i)$
  for every $\underline{\gamma}=\gamma_1\otimes\cdots \otimes
  \gamma_n$.
\end{notat}
\section{A virtual fundamental class}\label{section_virtualclass}
Let $D$ be a Dedekind domain and set $S = \spec D$. Let $\calX$ be a
smooth proper tame Deligne-Mumford stack of finite presentation over $S$,
admitting a projective coarse moduli scheme $X$. We want to define a
virtual fundamental class for $\KgnXbe$ relative to the forgetful
morphism \[\theta \colon \KgnXbe \rightarrow \MgnStw\text{,}\] following
the construction of \cite{BF} 7. For this, we need a
perfect relative obstruction theory for $\theta$.

\subsection{The stack of morphisms}\label{subsection_ot}
With notations as in Definition~\ref{def_morphisms}, notice that $\overline{\scrC}=\calC
\times_{\Mgntw}\underline{\Hom}_{\Mgntw}(\calC,\calX)$ is a universal
family for $\underline{\Hom}_{\Mgntw}(\calC,\calX)$ and we have the
following commutative diagram
  \[
  \begin{tikzpicture}
    \def\x{3.0}
    \def\y{-1.6}
    \node (A0_0) at (0*\x, 0*\y) {$\scrC$};
    \node (A0_1) at (1*\x, 0*\y) {$\overline{\scrC}$};
    \node (A0_2) at (2*\x, 0*\y) {$\calX$};
    \node (A1_0) at (0*\x, 1*\y) {$\KgnXbe$};
    \node (A1_1) at (1*\x, 1*\y) {$\underline{\Hom}_{\Mgntw}(\calC,\calX)$};
    \path (A0_0) edge [->] node [auto] {$\scriptstyle{}$} (A0_1);
    \path (A0_1) edge [->] node [below] {$\scriptstyle{\overline{\psi}}$} (A0_2);
    \path (A1_0) edge [->] node [auto] {$\scriptstyle{}$} (A1_1);
    \path (A0_0) edge [->, bend left=20] node [auto] {$\scriptstyle{\psi}$} (A0_2);
    \path (A0_0) edge [->] node [auto] {$\scriptstyle{\pi}$} (A1_0);
    \path (A0_1) edge [->] node [auto] {$\scriptstyle{\overline{\pi}}$} (A1_1);
  \end{tikzpicture}
  \]
\begin{lemma}\label{lemma_obstr_theory}
  We have $F^\bullet=
  R\overline{\pi}_*(\overline{\psi}^*\Omega_{\sfrac{\calX}{S}} \otimes
  \omega_{\overline{\pi}})[-1]\in D_\coh^{(-1,0)}\left(\underline{\Hom}_{\Mgntw}(\calC,\calX)\right)$ and
  $\sfrac{h^1}{h^0}(F^\bullet)$ is a vector bundle stack.
\end{lemma}
\begin{proof}
Since $\calX$ is smooth over $S$, the sheaf $\Omega_{\sfrac{\calX}{S}}$ is a
vector bundle over $\calX$. The dualizing sheaf $\omega_{\overline{\pi}}$
is a line bundle over $\overline{\scrC}$, because $\overline{\scrC}$ is
a family of curves with at most nodal singularities (which are
Gorenstein). Hence $\overline{\psi}^*\Omega_{\sfrac{\calX}{S}} \otimes
\omega_{\overline{\pi}}$ is a vector bundle on
$\overline{\scrC}$. Recall that the cohomology of the total
pushforward is the higher pushforward sheaf. Moreover,
$\overline{\pi}$ is a flat projective morphism of relative dimension
$1$, so the $i$-pushforward vanishes for $i>1$ by the cohomology and
base-change theorem (\cite{FGAexplained} Corollary~8.3.4), therefore
\[R\overline{\pi}_*(\overline{\psi}^*\Omega_{\sfrac{\calX}{S}} \otimes
\omega_{\overline{\pi}}) \in D_\coh^{(0,1)} \left(\underline{\Hom}_{\Mgntw}(\calC,\calX)\right)\text{.}\] Set
$\mathscr{F}=\overline{\psi}^*\Omega_{\sfrac{\calX}{S}} \otimes
\omega_{\overline{\pi}}$. Let $\scrL$ be a $\overline{\pi}$-ample line
bundle
 then, for $n$ big enough,
$\mathscr{F}\otimes \scrL^n$ is generated by global sections and
$R^0\overline{\pi}_*(\mathscr{F}\otimes \scrL^{-n})=0$. Thus we have a
surjection \[\mathscr{G}=\overline{\pi}^*\overline{\pi}_*(\mathscr{F}\otimes \scrL^n)\otimes
\scrL^{- n}\rightarrow \mathscr{F}\text{,}\] and we denote by
$\mathscr{K}$ the kernel. Notice that $\mathscr{K}$ is a vector bundle
because it is the kernel of a surjection of vector bundles. Hence we
get the following exact sequence \[0 \rightarrow
R^0\overline{\pi}_*\mathscr{K} \rightarrow
R^0\overline{\pi}_*\mathscr{G}\rightarrow
R^0\overline{\pi}_*\mathscr{F}\rightarrow
R^1\overline{\pi}_*\mathscr{K}\rightarrow
R^1\overline{\pi}_*\mathscr{G}\rightarrow
R^1\overline{\pi}_*\mathscr{F}\rightarrow 0\text{.}\] Since
$R^0\overline{\pi}_*(\mathscr{F}\otimes \scrL^{-n})=0$, we have that
$R^0\overline{\pi}_*\mathscr{G}=0$ and thus
$R^0\overline{\pi}_*\mathscr{K}=0$. As a consequence,
$R^1\overline{\pi}_*\mathscr{K}$ and $R^1\overline{\pi}_*\mathscr{G}$
are vector bundles and $F^\bullet$ is quasi-isomorphic to
$[R^1\overline{\pi}_*\mathscr{K}\rightarrow
R^1\overline{\pi}_*\mathscr{G}]$.
\end{proof}
\begin{parag}
We define a morphism $\barphi \colon F^\bullet \rightarrow \tau_{\geq
  -1} L_{\bartheta}^\bullet$ in $D_\coh^{(-1,0)}\left(\underline{\Hom}_{\Mgntw}(\calC,\calX)\right)$ as
follows. By adjuction, this is equivalent to define a
morphism \[(\overline{\psi}^*\Omega_{\sfrac{\calX}{S}} \otimes
\omega_{\overline{\pi}})[-1] \rightarrow
L\overline{\pi}^!(L_{\bartheta}^\bullet)\text{.}\] Recall that if
$\overline{\pi}$ is a flat proper Gorenstein morphism of relative
dimension $N$, then $L \overline{\pi}^!= \overline{\pi}^* \otimes
\omega_{\overline{\pi}}[-N]$. This applies in our case with $N=1$ and
we get $L \overline{\pi}^! = \overline{\pi}^* \otimes
\omega_{\overline{\pi}}[-1]$. Hence to give the morphism $\barphi$ is
equivalent to giving a morphism
$\overline{\psi}^*\Omega_{\sfrac{\calX}{S}} \rightarrow
\overline{\pi}^*L_{\bartheta}^\bullet$. Notice that
$\Omega_{\sfrac{\calX}{S}} = L_{\sfrac{\calX}{S}}^\bullet$, since $\calX$ is
smooth over $S$. We define the morphism
$\overline{\psi}^*L_{\sfrac{\calX}{S}}^\bullet \rightarrow
\overline{\pi}^*L_{\bartheta}^\bullet$ as the
composition \[\overline{\psi}^*L_{\sfrac{\calX}{S}}^\bullet \rightarrow
L_{\sfrac{\overline{\scrC}}{S}}^\bullet \rightarrow
L_{\sfrac{\overline{\scrC}}{\calC}}^\bullet\cong\overline{\pi}^*
L_{\bartheta}^\bullet\text{,}\] where $\calC$ is the universal curve of
$\MgnStw$, the isomorphism
$L_{\sfrac{\overline{\scrC}}{\calC}}^\bullet\cong\overline{\pi}^*
L_{\bartheta}^\bullet$ follows from the fact that
$\overline{\scrC}=\calC\times_{\Mgntw}\underline{\Hom}_{\Mgntw}(\calC,\calX)$ and the morphism $\calC
\rightarrow \MgnStw$ is flat (\cite{olsson} 8.1).
\end{parag}
\begin{prop}\label{prop_obstheory_F}
The pair $(F^\bullet, \barphi)$ defined above is a perfect relative
obstruction theory for $\bartheta$.
\end{prop}
\begin{proof}
  Let $\spec \bark\xrightarrow{\barx} \underline{\Hom}_{\Mgntw}(\calC,\calX)$ be a geometric
  point. Then $\barx$ corresponds to a twisted pointed curve $\calC_{\barx}$
  over $\bark$ together with a representable morphism $\overline{\psi}_{\barx}\colon
  \calC_{\barx}\rightarrow \calX$, obtained by pulling back
  $(\overline{\scrC},\overline{\psi})$ along $\barx$.
  Using Serre duality and cohomology and base chage theorem
  (\cite{FGAexplained} Corollary~8.3.4), we have
  \[ H^i(\calC_{\barx}, \overline{\psi}_{\barx}^*T_{\sfrac{\calX}{S}})
  =H^{1-i}(\calC_{\barx}, \barx^* (\overline{\psi}^*\Omega_{\sfrac{\calX}{S}}
  \otimes
  \omega_{\overline{\pi}}))\spcheck=h^{i-1}((F^\bullet[-1])\spcheck)=
  h^{i}((L \barx^* F^\bullet)\spcheck)\text{.}\] Now, let $A'
  \rightarrow A=\sfrac{A'}{I}$ be a small extension in
  $(\sfrac{\Art}{\hat{\scrO}_{S,\overline{s}}})$ and consider a
  commutative diagram
  \[
  \begin{tikzpicture}[xscale=2.6,yscale=-1.6]
    \node (A0_0) at (0, 0) {$\spec A$};
    \node (A0_1) at (1, 0) {$\underline{\Hom}_{\Mgntw}(\calC,\calX)$};
    \node (A1_0) at (0, 1) {$\spec A'$};
    \node (A1_1) at (1, 1) {$\MgnStw$};
    \path (A0_0) edge [->]node [auto] {$\scriptstyle{g}$} (A0_1);
    \path (A1_0) edge [->]node [auto] {$\scriptstyle{h'}$} (A1_1);
    \path (A0_1) edge [->]node [auto] {$\scriptstyle{\overline{\theta}}$} (A1_1);
    \path (A0_0) edge [->]node [left] {$\scriptstyle{i}$} (A1_0);
  \end{tikzpicture}
  \]
  In particular $h'$ corresponds to a family of twisted curves
  $\calC_{A'}$ over $A'$, obtained by pulling back $\calC \rightarrow \MgnStw$
  along $h'$, and $g$ corresponds to a family of twisted curves
  $\calC_A$ over $A$ together with a representable morphism $\overline{\psi}_A\colon \calC_A
  \rightarrow \calX$, obtained by pulling back
  $(\overline{\scrC},\overline{\psi})$ along $g$. Thus $g$ extends to
  $g'\colon \spec A'\rightarrow \underline{\Hom}_{\Mgntw}(\calC,\calX)$ if and only if
  $\overline{\psi}_A$ extends to $\overline{\psi}_{A'}\colon
  \calC_{A'}\rightarrow \calX$ if and only if, by deformation theory,
  $h^1(\overline{\phi}\spcheck)(\ob_{\overline{\theta}}(g,h'))$ is
  zero in $H^1(\calC_{\barx},
  \overline{\psi}_{\barx}^*T_{\sfrac{\calX}{S}})\otimes I$. Moreover the
  extensions, if they exist, form a torsor under $H^0(\calC_{\barx},
  \overline{\psi}_{\barx}^*T_{\sfrac{\calX}{S}})\otimes I$. By
  \cite{BF} 4.5,
  $(F^\bullet,\overline{\phi})$ is a relative obstruction theory for
  $\overline{\theta}$ and, by Lemma~\ref{lemma_obstr_theory},
  $F^\bullet$ is perfect.
\end{proof}
\subsection{A perfect obstruction theory for $\KgnXbe$}\label{subsection_vc}
\begin{cor}
  Let $E^\bullet =R \pi_*(\psi^* \Omega_{\sfrac{\calX}{S}} \otimes
  \omega_\pi)[-1]$ and let $\phi \colon
  E^\bullet \rightarrow \tau_{\geq -1} L_\theta^\bullet$ be the
  morphism induced by $\barphi$. Then $(E^\bullet, \phi)$ is a perfect
  relative obstruction theory for $\theta$.
\end{cor}
\begin{proof}
  Since the natural inclusion $j \colon \KgnXbe \hookrightarrow
  \underline{\Hom}_{\Mgntw}(\calC,\calX)$ is an open immersion, it follows that $Lj^*
  L_{\overline{\theta}}^\bullet = L_\theta^\bullet$, $Lj^*F^\bullet =
  E^\bullet$, and $\phi = j^* \barphi$. Hence, by
  Lemma~\ref{lemma_obstr_theory}, we have $E^\bullet \in
  D_\coh^{(-1,0)}\left(\KgnXbe\right)$. By Proposition~\ref{prop_obstheory_F}, we
  know that $(F^\bullet, \barphi)$ is a perfect obstruction theory for
  $\overline{\theta}$, hence $h^0(\barphi)$ is an isomorphism and
  $h^{-1}(\barphi)$ is surjective. Since the pullback $j^*$ is an
  exact functor, we have that $h^0(\phi)$ is an isomorphism and
  $h^{-1}(\phi)$ is surjective, which implies the statement.
\end{proof}
\begin{definit}
  We define the \emph{virtual fundamental class} of $\KgnXbe$ to
  be \[{[\KgnXbe]}^\virt = {[\KgnXbe, E^\bullet]}^\virt \in
  A_*(\sfrac{\KgnXbe}{S})\text{.}\] 
\end{definit}
\begin{remark}
Consider the vector bundle stack $\mu \colon
\frakE_\theta=\sfrac{h^1}{h^0}(E^\bullet) \rightarrow
\KgnXbe$. Then, for a geometric point $\barx$ of a component $\mathcal{K}$ of $\KgnXbe$, by Riemann-Roch theorem (\cite{AGV08} 7.2.1),
\begin{align*}
  \rk \barx^*\frakE_\theta &= \dim h^{-1}(L\barx^*E^\bullet) - \dim
  h^{0}(L\barx^*E^\bullet)\\ &= \dim H^1(\scrC_{\barx},
  \psi_{\barx}^*T_{\sfrac{\calX}{S}}) - \dim H^0(\scrC_{\barx},
  \psi_{\barx}^*T_{\sfrac{\calX}{S}})\\ &=
  (g-1)\rk(\psi_{\barx}^*T_{\sfrac{\calX}{S}}) -
  c_1(\psi_{\barx}^*T_{\sfrac{\calX}{S}})\cdot[\scrC_{\barx}] +\sum_{i=1}^n \age(\Sigma_i)\\
  &= (g-1)\dim_S \calX -
  c_1(T_{\sfrac{\calX}{S}})\cdot {\psi_{\barx}}_*[\scrC_{\barx}]+\sum_{i=1}^n \age(\Sigma_i)\text{,}
\end{align*}
where
$\age(\Sigma_i)=\age(\psi_{\barx}^*T_{\sfrac{\calX}{S}}|_{\Sigma_i})$
denotes the age of a locally free sheaf as defined in \cite{AGV08} 7.1
(recall that the age is constant on connected components of
$\inertia$). Thus the dimension of ${[\mathcal{K}]}^\virt$ is \[\dim_S \MgnStw -\rk \barx^*\frakE_\theta =(\dim_S
\calX - 3)(1-g) + c_1(T_{\sfrac{\calX_\eta}{\eta}})\cdot
{\psi_{\barx}}_*[\scrC_{\barx}]-\sum_{i=1}^n
\age(\Sigma_i) + n\text{.}\]
\end{remark}
\subsection{Properties}
\begin{parag}[\cite{AGV08} 5.1]
Let $\mathfrak{D}^{\text{tw}}(g_1,A|g_2,B)$ be the category fibered in
groupoids over $(\sfrac{\Sch}{S})$ which parametrizes nodal twisted
curves with a distinguished node separating the curve in two
components, one of genus $g_1$ containing the markings in a subset
$A\subset \{1, \ldots, n\}$, the other of genus $g_2$ containing the
markings in the complementary set $B$.  The category
$\mathfrak{D}^{\text{tw}}(g_1,A|g_2,B)$ is a smooth algebraic stack,
locally of finite presentation over $S$. Let $g=g_1+g_2$, there is a
natural representable morphism \[\gl \colon
\mathfrak{D}^{\text{tw}}(g_1,A|g_2,B) \rightarrow \Mgntw\] induced by
gluing the two families of curves into a family of reducible curves
with a distinguished node.
\end{parag}
\begin{prop}\label{prop_gl1}
\begin{enumerate}
\item Consider the evaluation morphisms
  $\check{e}_{\check{\bullet}}\colon\!\mathcal{K}_{g_1,A\sqcup
    \check{\bullet}}(\calX, \beta_1\!) \rightarrow
  \overline{\mathcal{I}}_\mu(\calX)$ and $e_\bullet\colon
  \mathcal{K}_{g_2,B\sqcup \bullet}(\calX, \beta_2) \rightarrow
  \overline{\mathcal{I}}_\mu(\calX)$.  There exists a natural
  representable morphism \[\mathcal{K}_{g_1,A\sqcup
    \check{\bullet}}(\calX,
  \beta_1)\times_{\overline{\mathcal{I}}_\mu(\calX)}\mathcal{K}_{g_2,B\sqcup
    \bullet}(\calX, \beta_2) \rightarrow \mathcal{K}_{g_1+g_2,A\sqcup
    B}(\calX, \beta_1+\beta_2)\text{.}\]
\item Consider the evaluation morphisms
  $\check{e}_{\check{\bullet}}\times e_{\bullet}
  \colon\mathcal{K}_{g-1,A\sqcup \{\check{\bullet},\bullet\}}(\calX,
  \beta_\eta) \rightarrow\overline{\mathcal{I}}_\mu(\calX)^{2}$ and
  the diagonal $\Delta \colon
  \overline{\mathcal{I}}_\mu(\calX)\rightarrow
  \overline{\mathcal{I}}_\mu(\calX)^{2}$ (\ref{delta}).  There exists a natural
  representable morphism \[\mathcal{K}_{g-1,A\sqcup
    \{\check{\bullet},\bullet\}}(\calX,
  \beta_\eta)\times_{\overline{\mathcal{I}}_\mu(\calX)^{2}}\overline{\mathcal{I}}_\mu(\calX)
  \rightarrow \mathcal{K}_{g,A}(\calX, \beta_\eta)\text{.}\]
\item We have a cartesian diagram
  \[
  \begin{tikzpicture}[xscale=6.0,yscale=-2.4]
    \node (A0_0) at (0, 0) {$\displaystyle{\bigsqcup_{\beta_1+\beta_2=\beta_\eta}\mathcal{K}_{g_1,A\sqcup \check{\bullet}}(\calX, \beta_1)\times_{\overline{\mathcal{I}}_\mu(\calX)}\mathcal{K}_{g_2,B\sqcup \bullet}(\calX, \beta_2)}$};
    \node (A0_1) at (1, 0) {$\mathcal{K}_{g_1+g_2,A\sqcup B}(\calX, \beta_\eta)$};
    \node (A1_0) at (0, 1) {$\mathfrak{D}^{\text{tw}}(g_1,A|g_2,B)$};
    \node (A1_1) at (1, 1) {$\mathfrak{M}_{g_1+g_2,A\sqcup B}^{\text{tw}}$};
    \path (A0_0) edge [->]node [auto] {$\scriptstyle{}$} (A0_1);
    \path (A0_0) edge [->]node [auto] {$\scriptstyle{}$} (A1_0);
    \path (A0_1) edge [->]node [auto] {$\scriptstyle{}$} (A1_1);
    \path (A1_0) edge [->]node [auto] {$\scriptstyle{\gl}$} (A1_1);
  \end{tikzpicture}
  \]
\end{enumerate}
\end{prop}
\begin{proof}
Follows in the same way as in \cite{AGV08} 5.2.
\end{proof}
\begin{parag}
By \cite{AGV08} 6.2.4, the morphism $\gl$ is finite and unramified,
therefore, by \cite{kresch} 4.1, it induces a pull-back homomorphism
on Chow groups \[\gl^!\colon A_*(\KgnXbe)\rightarrow
\bigoplus_{\beta_1+\beta_2=\beta_\eta}A_*(\mathcal{K}_{g_1,A\sqcup
  \check{\bullet}}(\calX,
\beta_1)\times_{\overline{\mathcal{I}}_\mu(\calX)}\mathcal{K}_{g_2,B\sqcup
  \bullet}(\calX, \beta_2))\text{.}\]
\end{parag}
\begin{prop}\label{prop_gl}
Consider the diagonal $\Delta \colon
  \overline{\mathcal{I}}_\mu(\calX)\rightarrow
  \overline{\mathcal{I}}_\mu(\calX)^{2}$ (\ref{delta}). We have
\begin{enumerate}
\item $\gl^!{[\mathcal{K}_{g,A\sqcup B}(\calX,
\beta_\eta)]}^{\virt}\!=\!\sum_{\beta_1+\beta_2=\beta_\eta}\Delta^!({[\mathcal{K}_{g_1,A\sqcup
  \check{\bullet}}(\calX,
\beta_1)]}^\virt\times {[\mathcal{K}_{g_2,B\sqcup
  \bullet}(\calX, \beta_2)]}^\virt)$;
\item $\gl^!{[\mathcal{K}_{g,A}(\calX, \beta_\eta)]}^{\virt}=\Delta^!{[\mathcal{K}_{g-1,A\sqcup \{\check{\bullet},\bullet\}}(\calX, \beta_\eta)]}^\virt$.
\end{enumerate}
\end{prop}
\begin{proof}
  For the first part, by Proposition~\ref{prop_gl1} and
  \cite{BF} 7.2, \[\gl^!{[\mathcal{K}_{g,A\sqcup
      B}(\calX,
    \beta_\eta)]}^{\virt}=\sum_{\beta_1+\beta_2=\beta_\eta}{[\mathcal{K}_{g_1,A\sqcup
      \check{\bullet}}(\calX,
    \beta_1)\times_{\inertia}\mathcal{K}_{g_2,B\sqcup \bullet}(\calX,
    \beta_2)]}^\virt\text{.}\] Let us denote for simplicity
  $\mathcal{K}^{(1)}=\mathcal{K}_{g_1,A\sqcup \check{\bullet}}(\calX,
  \beta_1)$ and $\mathcal{K}^{(2)}=\mathcal{K}_{g_2,B\sqcup
    \bullet}(\calX, \beta_2)$. Let $E_j^\bullet$ be the perfect
  obstruction theory of $\mathcal{K}^{(j)}$ as constructed in
  section~\ref{subsection_vc}, then $E_1^\bullet\oplus E_2^\bullet$ is
  the perfect obstruction theory of
  $\mathcal{K}^{(1)}\times_k\mathcal{K}^{(2)}$. Let $E_{1,2}^\bullet$
  be the perfect obstruction theory of
  $\mathcal{K}^{(1)}\times_{\inertia}\mathcal{K}^{(2)}$. By
  considering the normalization sequence for a family of nodal curves
  with a distinguished node $\Sigma$ over
  $\mathcal{K}^{(1)}\times_{\inertia}\mathcal{K}^{(2)}$, we get the
  following distinguished triangle, as in \cite{AGV08}
  5.3.1, \[E_{1,2}^\bullet \rightarrow E_1^\bullet\oplus E_2^\bullet
  \rightarrow E_\Sigma^\bullet\text{,}\] where $E_{\Sigma}^\bullet$
  can be identified with the cotangent complex of the map $\Delta$ in
  the same way as in \cite{AGV08} 3.6.1. Then, by
  \cite{BF} 7.5, we
  get \[\Delta^!({[\mathcal{K}_{g_1,A\sqcup \check{\bullet}}(\calX,
    \beta_1)]}^\virt\times {[\mathcal{K}_{g_2,B\sqcup \bullet}(\calX,
    \beta_2)]}^\virt)\!=\!{[\mathcal{K}_{g_1,A\sqcup
      \check{\bullet}}(\calX,
    \beta_1)\!\times_{\inertia}\!\mathcal{K}_{g_2,B\sqcup \bullet}(\calX,
    \beta_2)]}^\virt\text{.}\] For the second part of the statement,
  we observe that, since $\Delta$ is a regular
  embedding, \[\Delta^!{[\mathcal{K}_{g-1,A\sqcup
      \{\check{\bullet},\bullet\}}(\calX,
    \beta_\eta)]}^\virt={[\mathcal{K}_{g-1,A\sqcup
      \{\check{\bullet},\bullet\}}(\calX,
    \beta_\eta)\times_{\inertia^2}\inertia]}^\virt\text{,}\] and, by
  \cite{BF} 7.2, the right-hand side is equal
  to $\gl^!{[\mathcal{K}_{g,A}(\calX, \beta_\eta)]}^{\virt}$.
\end{proof}
\section{Gromov-Witten classes and invariants}\label{section_GW}
\subsection{Gromov-Witten classes}
Let $D$ be a Dedekind domain, set $S=\spec D$ and denote by $\eta$ the
generic point of $S$. Let $\calX$ be a smooth proper tame
Deligne-Mumford stack of finite presentation over $S$, admitting a
projective coarse moduli scheme $X$. Set $X_\eta= X \times_S
\eta$. Fix $\beta_\eta\in A_1(\sfrac{X_\eta}{\eta})$ and $g, n \geq
0$ with $2g+n\geq 3$.
\begin{remark}
If $S= \spec k$ with $k$ an algebraically closed field and if $l$ is a
prime different from the characteristic of $k$, we can define the
$l$-adic \'etale cohomology as \[H^r(\inertia, \bbZ_l)= \varprojlim_m H_{\et}^r(\inertia,
\sfrac{\bbZ}{l^m \bbZ})\text{.}\] Moreover $H^r(\inertia, \bbQ_l)= H^r(\inertia,
\bbZ_l)\otimes_{\bbZ_l}\bbQ_l$ and we have the cycle map \[\cl \colon
{A^r(\sfrac{\inertia}{k})}_{\bbQ} \rightarrow H^{2r}(\inertia, \bbQ_l(r))\] as
described in \cite{milne} VI.9. We set $H^*(\inertia)= \sum_r H^r(\inertia,
\bbQ_l(\overline{r}))$, where $\overline{r}$ is the integral part of
$\sfrac{r}{2}$.
\end{remark}
\begin{definit}[Gromov-Witten classes]
  We define the linear operator \[\Ignbe \colon
  {A^*(\sfrac{\inertia}{S})}_{\bbQ}^{\otimes n} \rightarrow
  {A^*(\sfrac{\mgnS}{S})}_{\bbQ}\] such that, given $\underline{\gamma}\in
  {A^*(\sfrac{\inertia}{S})}_{\bbQ}^{\otimes n}$, \[\Ignbe(\gamma_1 \otimes \cdots
  \otimes \gamma_n) = q_*\left(e^*(\underline{\gamma}) \cap
  {[\mathcal{K}_{g,n}(\sfrac{\calX}{S},
    \beta_\eta)]}^\virt\right)\text{,}\] where $q\colon
\KgnXbe\rightarrow \mgnS$ forgets the map to $\calX$, passes to the
coarse curve and stabilizes. If moreover $S= \spec k$ with $k$
an algebraically closed field, we can define
\[\Ignbe \colon {H^*(\inertia)}^{\otimes n} \rightarrow H^*(\mgnS)\] as above, where, abusing the notation, we
write ${[\mathcal{K}_{g,n}(\sfrac{\calX}{S}, \beta_\eta)]}^\virt$
instead of the corresponding homology class $\cl
\left({[\mathcal{K}_{g,n}(\sfrac{\calX}{S},
    \beta_\eta)]}^\virt\right)$.
\end{definit}
\begin{definit}[Gromov-Witten invariants]
  We define \[\langle \Ignbe\rangle (\underline{\gamma})=
  \int_{\KgnXbe}\left(e^*(\underline{\gamma}) \cap
    {[\mathcal{K}_{g,n}(\sfrac{\calX}{S},
      \beta_\eta)]}^\virt\right)\text{,}\] for
  $\underline{\gamma}=\gamma_1 \otimes \cdots \otimes \gamma_n \in
  {A^*(\sfrac{\inertia}{S})}_{\bbQ}^{\otimes n}$.  If $S= \spec k$
  with $k$ an algebraically closed field then $\langle \Ignbe\rangle
  (\underline{\gamma})$ is defined for every $\underline{\gamma} \in
  {H^*(\inertia)}^{\otimes n}$.
\end{definit}
\begin{notat}
  When $S=\spec k$, we have $X_\eta=X$ and hence we will simply write
  $\beta$ instead of $\beta_\eta$.
\end{notat}
\begin{remark}
We have that
\begin{align*}
  \int_{\mgnS}\Ignbe (\underline{\gamma})&= \int_{\mgnS}q_*\left(e^*(\underline{\gamma}) \cap
    {[\mathcal{K}_{g,n}(\sfrac{\calX}{S}, \beta_\eta)]}^\virt\right)\\
  &= \int_{\mathcal{K}_{g,n}(\sfrac{\calX}{S},
    \beta_\eta)}\left(e^*(\underline{\gamma}) \cap
    {[\mathcal{K}_{g,n}(\sfrac{\calX}{S}, \beta_\eta)]}^\virt\right)\\
  &= \langle \Ignbe\rangle (\underline{\gamma}) \text{.}
\end{align*}
\end{remark}
\begin{definit}[\cite{AGV08} 6.1.1]
  We define a locally constant function $\rank\colon
  \overline{\mathcal{I}}_{\mu}(\calX) \rightarrow \bbZ$ by evaluating
  on geometric points, $\rank(\barx, \mathcal{G})=r$, where $\mathcal{G}$
  is a gerbe banded by $\mu_r$. We can view $\rank$ as an element of
  $A^0(\overline{\mathcal{I}}_{\mu}(\calX))$.
\end{definit}
\begin{parag}[Alternative definition]
Following the
  formalism of \cite{AGV08}, we could define $\Ignbe$ as a linear
  operator ${A^*(\sfrac{\inertia}{S})}_{\bbQ}^{\otimes n} \rightarrow
  {A^*(\sfrac{\inertia}{S})}_{\bbQ}$ such that \[\Ignbe(\gamma_1
  \otimes \cdots \otimes \gamma_n) = \rank \cdot \check{e}_{n+1
    *}\left(\left(\bigsqcap_{i=1}^ne_i^*(\gamma_i)\right) \cap
    {[\mathcal{K}_{g,n+1}(\sfrac{\calX}{S},
      \beta_\eta)]}^\virt\right)\text{.}\] With this definition,
\begin{align*}
  \int_{\inertia}&\frac{1}{\rank}\text{I}_{g,n-1,\beta_\eta}^\calX
  (\gamma_1\otimes \cdots \otimes \gamma_{n-1})\cap \iota^*(\gamma_n)\\
  &= \int_{\inertia}\check{e}_{n *}\left(\left(\bigsqcap_{i=1}^{n-1}e_i^*(\gamma_i)\right) \cap
    {[\mathcal{K}_{g,n}(\sfrac{\calX}{S},
      \beta_\eta)]}^\virt\right)\cap \iota^*(\gamma_n)\\
  &= \int_{\inertia}\check{e}_{n *}\left(\left(\bigsqcap_{i=1}^{n-1}e_i^*(\gamma_i)\right) \cap
    {[\mathcal{K}_{g,n}(\sfrac{\calX}{S},
      \beta_\eta)]}^\virt\cap \check{e}_{n+1}^*\iota^*(\gamma_n)\right)\\
  &= \int_{\mathcal{K}_{g,n}(\sfrac{\calX}{S},
        \beta_\eta)}\left(\left(\bigsqcap_{i=1}^{n-1}e_i^*(\gamma_i)\right) \cap
      {[\mathcal{K}_{g,n}(\sfrac{\calX}{S},
        \beta_\eta)]}^\virt\cap e_{n}^*(\gamma_n)\right)\\
  &= \langle \Ignbe\rangle (\underline{\gamma}) \text{.}
\end{align*}
\end{parag}
\begin{remark}
  Let $\calM$ be a proper Artin stack over a field $k$. Let $L$ be a
  finite algebraic extension of $k$, then
  $\calM_L=\calM\times_kL\xrightarrow{\rho_L} \calM$ is smooth and
  finite of degree $[L : k]$. By \cite{fulton} 1.7.4,
  ${\rho_L}_*\rho_L^*= [L:k]$, therefore $\rho_L^*$ gives an
  isomorphism ${A_*(\sfrac{\calM}{k})}_{\bbQ}\cong
  {A_*(\sfrac{\calM_L}{L})}_{\bbQ}$. Let $\bark$ be an algebraic
  closure of $k$ and set $\overline{\calM}=\calM\times_k\bark$, then
  $A_*(\sfrac{\overline{\calM}}{\bark})=
  \varinjlim_LA_*(\sfrac{\calM_L}{L})$, where the limit is over all
  finite algebraic extensions $L$ of $k$ such that $L \subset
  \bark$. There is an induced homomorphism $\rho \colon
  A_*(\sfrac{\calM}{k}) \rightarrow
  A_*(\sfrac{\overline{\calM}}{\bark})$ which gives an isomorphism
  ${A_*(\sfrac{\calM}{k})}_{\bbQ}\cong
  {A_*(\sfrac{\overline{\calM}}{\bark})}_{\bbQ}$; for all $\beta \in
  A_*(\sfrac{\calM}{k})$ we set $\overline{\beta}=\rho(\beta)$. The
  same holds for bivariant Chow groups ${A^*(\bullet)}_{\bbQ}$.
\end{remark}
\begin{prop}
  Let $\calX$ be a smooth proper tame Deligne-Mumford stack of finite
  presentation over a field $k$, admitting a projective coarse moduli
  scheme $X$, and set $\overline{\calX}=\calX\times_k\bark$. Then, for all
  $\underline{\gamma}\in {A^*(\sfrac{\inertia}{k})}_{\bbQ}^{\otimes n}$, \[\Ignb (\underline{\gamma})=\I_{g,n,\overline{\beta}}^{\overline{\calX}}(\underline{\gamma})\text{.}\]
\end{prop}
\begin{proof}
  Let $L$ be a finite algebraic extension of $k$ and set
  $\calX_L=\calX\times_kL$. Let $\beta_L=\rho_L^*\beta$. Notice that
  $\mathcal{K}_{g,n}(\sfrac{\calX_L}{L}, \beta_L)\cong \KgnXbk
  \times_k L$ and thus, by
  \cite{BF} 7.2, \[{[\KgnXbk]}^{\virt}=
  {[\mathcal{K}_{g,n}(\sfrac{\calX_L}{L},\beta_L)]}^{\virt}\in
  {A_*(\sfrac{\KgnXbk}{k})}_{\bbQ}\cong
  {A_*(\sfrac{\mathcal{K}_{g,n}(\sfrac{\calX_L}{L},\beta_L)}{L})}_{\bbQ}\text{.}\]
  Then for all $\underline{\gamma}\in {A^*(\sfrac{\inertia}{k})}_{\bbQ}^{\otimes n}$, we
  have $\I_{g,n,\beta_L}^{\calX_L} (\underline{\gamma}) =
  \Ignb(\underline{\gamma})$ and therefore, passing to the limit, we get $\I_{g,n,\overline{\beta}}^{\overline{\calX}}
  (\underline{\gamma}) = \Ignb(\underline{\gamma})$.
\end{proof}
\subsection{Comparison of invariants in mixed characteristic}
Let $D$ be a Dedekind domain, set $B = \spec D$. We denote by
$\eta=\spec K$ the generic point of $B$ and let $b_0,b_1\in B$ be
closed points of $B$. Let $\pi \colon \calY \rightarrow B$ be a smooth
proper tame Deligne-Mumford stack of finite presentation over $B$,
admitting a projective coarse moduli scheme $Y$ and set $\calY_\eta=\calY
\times_B\eta$, $\calY_h=\calY\times_B b_h$ for $h=0,1$. By \cite{fulton} 20.3,
there are specialization morphisms $\sigma_h \colon
A_*(\sfrac{\calY_\eta}{\eta})\rightarrow A_*(\sfrac{\calY_h}{b_h})$ for
$h=0,1$. Let $b_h =\spec k_h$ and let $\bark_h$ be an algebraic
closure of $k_h$ for $h=0,1$. We set $\overline{b}_h=\spec
\bark_h$. Recall that the cospecialization map gives an isomorphism
$H^*(\overline{\mathcal{I}}_\mu(\overline{\calY}_0))\cong H^*(\overline{\mathcal{I}}_\mu(\overline{\calY}_1))$, where
$\overline{\calY}_h=\calY_h\times_{k_h}\bark_h$ for $h=0,1$ (\cite{milne}
VI.4.1).
\begin{theorem}\label{theorem_definv}
  Let $\beta \in A_1(\sfrac{Y_\eta}{\eta})$ and set
  $\beta_h=\sigma_h(\beta)$ for
  $h=0,1$. Then \[\text{I}_{g,n,\overline{\beta}_0}^{\overline{\calY}_0}(\underline{\gamma})
  =
  \text{I}_{g,n,\overline{\beta}_1}^{\overline{\calY}_1}(\underline{\gamma})\text{,}\]
  for every $\underline{\gamma}\in {H^*(\overline{\mathcal{I}}_\mu(\overline{\calY}_0))}^{\otimes
    n}\cong {H^*(\overline{\mathcal{I}}_\mu(\overline{\calY}_1))}^{\otimes n}$.
 \end{theorem}
\begin{proof}
  Let $R_h$ be the localization of $D$ at $b_h$ for $h=0,1$, then
  $R_h$ is a discrete valuation ring with generic point $\eta$ and
  closed point $b_h$. Let $\hat{R}_h$ be the completion of $R_h$, then
  $\hat{R}_h$ is a complete discrete valuation ring with closed point
  $b_h$ and generic point $\eta \times_{R_h}\hat{R}_h$. Moreover
  $R_0\otimes_DR_1=K$ and hence $\eta \times_{R_0}\hat{R}_0 = \eta
  \times_{R_1}\hat{R}_1$. We denote by $\hat{\eta}=\spec \hat{K}$ the
  generic point of $\hat{R}_h$. Set $\hat{\calY}_h=\calY \times_D\hat{R}_h$
  and $\hat{\calY}_\eta= \calY \times_D \hat{\eta}$.  Let $i_h \colon \calY_h
  \rightarrow \hat{\calY}_h$ and $j_h \colon \hat{\calY}_\eta \rightarrow
  \hat{\calY}_h$ be the natural inclusions. Let $\hat{\beta}\in
  A_1(\sfrac{\hat{Y}_\eta}{\hat{\eta}})$ be the pullback of
  $\beta$. We have the following cartesian diagram
  \[
  \begin{tikzpicture}[xscale=3.2,yscale=-1.8]
    \node (A0_0) at (0, 0) {$\mathcal{K}_{g,n}(\sfrac{\hat{\calY}_\eta}{\hat{\eta}}, \hat{\beta})$};
    \node (A0_1) at (1, 0) {$\mathcal{K}_{g,n}(\sfrac{\hat{\calY}_h}{\hat{R}_h}, \hat{\beta})$};
    \node (A0_2) at (2, 0) {$\mathcal{K}_{g,n}(\sfrac{\calY_h}{b_h}, \beta_h)$};
    \node (A1_0) at (0, 1) {$\mathfrak{M}_{\sfrac{g,n}{\hat{\eta}}}$};
    \node (A1_1) at (1, 1) {$\mathfrak{M}_{\sfrac{g,n}{\hat{R}_h}}$};
    \node (A1_2) at (2, 1) {$\mathfrak{M}_{\sfrac{g,n}{b_h}}$};
    \path (A0_1) edge [->]node [auto] {$\scriptstyle{}$} (A1_1);
    \path (A0_0) edge [->]node [auto] {$\scriptstyle{\hat{\j}}$} (A0_1);
    \path (A1_2) edge [->]node [above] {$\scriptstyle{\widetilde{\i}}$} (A1_1);
    \path (A1_0) edge [->]node [auto] {$\scriptstyle{\widetilde{\j}}$} (A1_1);
    \path (A0_2) edge [->]node [auto] {$\scriptstyle{}$} (A1_2);
    \path (A0_2) edge [->]node [above] {$\scriptstyle{\hat{\i}}$} (A0_1);
    \path (A0_0) edge [->]node [auto] {$\scriptstyle{}$} (A1_0);
  \end{tikzpicture}
  \]
  Let $\overline{K}$ be an algebraic closure of $\hat{K}$. We set
  $\overline{\beta}=\rho(\hat{\beta})\in
  A_1(\sfrac{\overline{Y}_\eta}{\overline{\eta}})$, where
  $\overline{\eta}=\spec \overline{K}$ and
  $\overline{\calY}_\eta=\calY\times_D\overline{\eta}$. By \cite{fulton}
  20.3.5 and \cite{kresch} 3.5.7, 5.3.1, there exists a specialization
  homomorphism \[\hat{\sigma}_h \colon
  A_*(\sfrac{\mathcal{K}_{g,n}(\sfrac{\overline{\calY}_\eta}{\overline{\eta}},
    \overline{\beta})}{\overline{\eta}})_{\bbQ}\rightarrow
  A_*(\sfrac{\mathcal{K}_{g,n}(\sfrac{\overline{\calY}_h}{\overline{b}_h},
    \overline{\beta}_h)}{\overline{b}_h})_{\bbQ}\text{,}\] and, by the
  functoriality of the virtual fundamental class
  (\cite{BF} 7.2),
  \[\hat{\sigma}_h({[\mathcal{K}_{g,n}(\sfrac{\overline{\calY}_\eta}{\overline{\eta}},
    \overline{\beta})]}^\virt)={[\mathcal{K}_{g,n}(\sfrac{\overline{\calY}_h}{\overline{b}_h},\overline{\beta}_h)]}^\virt\text{.}\]
  By \cite{milne} VI.4.1, there are isomorphisms
  $H^*(\overline{\mathcal{I}}_\mu(\overline{\calY}_\eta))\cong H^*(\overline{\mathcal{I}}_\mu(\overline{\calY}_h))$ for $h=0,1$,
  compatible with evaluation maps. It follows
  that, for $h=0,1$, 
  $\text{I}_{g,n,\overline{\beta}_h}^{\overline{\calY}_h}(\underline{\gamma})
  =
  \text{I}_{g,n,\overline{\beta}}^{\overline{\calY}_\eta}(\underline{\gamma})$ for
  $\underline{\gamma} \in {H^*(\overline{\mathcal{I}}_\mu(\overline{\calY}_h))}^{\otimes n}$.
\end{proof}
\begin{cor}
  Let $\calX$ be a smooth proper tame Deligne-Mumford stack of finite
  presentation over a field $k$, admitting a projective coarse moduli
  scheme $X$. Then the Gromov-Witten invariants $\langle\Ignb\rangle$
  are invariant under deformations of $\calX$.
\end{cor}
\subsection{Axioms}
Let $\calX$ be a smooth proper tame Deligne-Mumford stack of finite
presentation over an algebraically closed field $k$, admitting a
projective coarse moduli scheme $X$.
\subsubsection{Effectivity}
Let ${A_1(\sfrac{X}{k})}_+$ be the set of $\beta \in
A_1(\sfrac{X}{k})$ such that $\beta \cdot c_1(\scrL)\geq
0$ for every ample line bundle $\scrL$. Then $\Ignb = 0$, for all
$\beta \notin {A_1(\sfrac{X}{k})}_+$.
\begin{proof}
  If $\KgnXbk \neq \emptyset$ then $\beta= f_*[C]$ for some
  stable map $(C, x_i, f)$, hence $\beta \in
  {A_1(\sfrac{X}{k})}_+$. It follows that $\KgnXbk
  = \emptyset$ for every
  $\beta \notin {A_1(\sfrac{X}{k})}_+$, and thus ${[\KgnXbk]}^\virt = 0$.
\end{proof}
\subsubsection{$S_n$-covariance}
For all $\gamma_j\in H^{m_j}(\inertia)$, we have \[\Ignb (\gamma_1 \otimes\!\cdots\!\otimes \gamma_i \otimes \gamma_{i+1}\otimes\!\cdots\!\otimes
\gamma_n)\!=\!{(-1)}^{m_im_{i+1}} \Ignb (\gamma_1 \otimes\!\cdots\!\otimes
\gamma_{i+1} \otimes \gamma_{i}\otimes\!\cdots\!\otimes
\gamma_n)\text{.}\]
\begin{proof}
The statement follows from the following (\cite{milne} VI.8) \[\gamma_1 \otimes
\cdots\otimes \gamma_i \otimes \gamma_{i+1}\otimes \cdots \otimes
\gamma_n\!=\!{(-1)}^{m_im_{i+1}} \gamma_1 \otimes \cdots\otimes
\gamma_{i+1} \otimes \gamma_{i}\otimes \cdots \otimes
\gamma_n\in H^*(\inertia^n)\text{.}\qedhere\]
\end{proof}
\subsubsection{Grading}
Let us set $H_{\st}^*(\calX)=H^*(\inertia)$. We consider
$H_\st^*(\calX)$ as a graded group with the following grading
$H_\st^m(\calX)=\bigoplus_{\Omega}H^{m-2\age(\Omega)}(\Omega)$, where
the sum is taken over all connected components $\Omega$ of $\inertia$.
We have \[\Ignb \colon \bigotimes_{i=1}^n{H_\st^{m_i}(\calX)}
\rightarrow H^{\sum m_i + 2((g-1)\dim_k \calX -
  c_1(T_{\sfrac{\calX}{k}})\cdot \beta)}(\mgnk)\text{.}\]
\begin{proof}
  Let $\barx=(\calC\rightarrow \calX,\Sigma_1,\ldots, \Sigma_n)$ be a
  geometric point of a component $\mathcal{K}$ of
  $\mathcal{K}_{g,n}(\calX, \beta)$ then, for $i=1, \ldots, n$, we
  have evaluation maps $e_i\colon \mathcal{K}\rightarrow \Omega_i$ for
  connected components $\Omega_j$ of $\inertia$. Since the age only
  depends on the connected component, we have
  $\age(\Sigma_i)=\age(\Omega_i)$. The virtual fundamental class
  ${[\mathcal{K}]}^\virt$ is a cycle class of dimension \[(\dim_S
  \calX - 3)(1-g) + c_1(T_{\sfrac{\calX}{k}})\cdot \beta-\sum_{i=1}^n
  \age(\Sigma_i) + n\text{.}\] Notice that $\gamma_i \in
  H_{\st}^{m_i}(\calX)=H^{m_i-2\age(\Omega_i)}(\Omega_i)$. It follows
  that $\text{I}_{g,n,\beta,\Omega_{n+1}}^\calX(\underline{\gamma})$ has degree
\begin{multline*}
2(3g\!-\!3\!+\!n)\!-\!2\left(\!(\dim_k \calX \!-\! 3)(1\!-\!g)\!+\!n \!+\!
  c_1(T_{\sfrac{\calX}{k}})\!\cdot\! \beta\!-\!\sum_{i=1}^n \age(\Sigma_i)\!\right)\!+\\+\!\sum_{i=1}^n (m_i\!-\!2\age(\Omega_i))
=\sum_{i=1}^n m_i +2((g-1)\dim_k \calX -
  c_1(T_{\sfrac{\calX}{k}})\cdot \beta)\text{.}\qedhere
\end{multline*}
\end{proof}
\subsubsection{Fundamental class} Let $\phi_n
\colon \overline{\mathcal{M}}_{\sfrac{g,n+1}{k}} \rightarrow \mgnk$ be
the natural functor that forgets the last marked point and
stabilizes. We have
\begin{align*}
\text{I}_{g, n+1, \beta}^\calX(\bullet \otimes \id) &= \phi_n^*\Ignb(\bullet)\text{,}\\
\text{I}_{0, 3, \beta}^\calX(\gamma_1 \otimes \gamma_2 \otimes \id) &=
\begin{cases}
\int_{\inertia}\frac{1}{\rank}\gamma_1 \cup \iota^*\gamma_2& \text{if }\beta=0\\
0 & \text{otherwise.}
\end{cases}
\end{align*}
\begin{proof}
Let us form the cartesian diagram
  \[
  \begin{tikzpicture}[xscale=2.2,yscale=-1.6]
    \node (A0_0) at (0, 0) {$\calM$};
    \node (A0_1) at (1, 0) {$\calN$};
    \node (A0_2) at (2, 0) {$\KgnXbk$};
    \node (A1_0) at (0, 1) {$\mathfrak{M}_{\sfrac{g, n+1}{k}}^{\text{tw}}$};
    \node (A1_1) at (1, 1) {$\mathfrak{N}$};
    \node (A1_2) at (2, 1) {$\Mgnktw$};
    \node (A2_1) at (1, 2) {$\overline{\calM}_{\sfrac{g, n+1}{k}}$};
    \node (A2_2) at (2, 2) {$\mgnk$};
    \path (A0_1) edge [->]node [auto] {$\scriptstyle{\widetilde{\theta}}$} (A1_1);
    \path (A0_0) edge [->]node [auto] {$\scriptstyle{j}$} (A0_1);
    \path (A2_1) edge [->]node [auto] {$\scriptstyle{\phi_n}$} (A2_2);
    \path (A1_0) edge [->]node [auto] {$\scriptstyle{}$} (A1_1);
    \path (A1_1) edge [->]node [auto] {$\scriptstyle{}$} (A1_2);
    \path (A0_2) edge [->]node [auto] {$\scriptstyle{\theta_n}$} (A1_2);
    \path (A1_1) edge [->]node [auto] {$\scriptstyle{}$} (A2_1);
    \path (A0_0) edge [->]node [auto] {$\scriptstyle{\hat{\theta}}$} (A1_0);
    \path (A0_1) edge [->]node [auto] {$\scriptstyle{\widetilde{\phi}}$} (A0_2);
    \path (A1_2) edge [->]node [auto] {$\scriptstyle{}$} (A2_2);
  \end{tikzpicture}
  \]
  and notice that $\calM$ is the algebraic stack of twisted stable maps of
  genus $g$ and class $\beta$ with $n+1$ gerbes which
  remain stable if we forget the last gerbe. In particular
  there is a regular embedding $i \colon \calM \rightarrow
  \mathcal{K}_{g, n+1}(\sfrac{\calX}{k}, \beta)$ which commute
  with $\theta_{n+1}$ and $\hat{\theta}$.  If we define a virtual
  fundamental class ${[\mathcal{M}]}^\virt$ relative to $\hat{\theta}$
  as described in section~\ref{subsection_vc} then
  \[i^!{[\mathcal{K}_{g,n+1}(\sfrac{\calX}{k}, \beta)]}^\virt =
  {[\mathcal{M}]}^\virt\text{.}\] If we define a virtual fundamental
  class ${[\mathcal{N}]}^\virt$ relative to $\widetilde{\theta}$ then,
  by
  \cite{BF} 7.2, \[j^*\widetilde{\phi}^*{[\KgnXbk]}^\virt
  = j^*{[\mathcal{N}]}^\virt= {[\calM]}^{\virt}\text{.}\] Let
  $\widetilde{q}\colon \mathcal{N}\rightarrow
  \overline{\mathcal{M}}_{\sfrac{g,n+1}{k}}$ and let $\pi \colon
  \inertia^{n+1} \rightarrow \inertia^{n}$ be the projection on the
  first $n$ components. Moreover we denote
  $\widetilde{e}=e_{(n)}\circ \widetilde{\phi}$,
  $\hat{e}=e_{(n+1)}\circ i$ and observe that $q_{n+1}\circ i=\widetilde{q}\circ j$. We have that
\begin{align*}
\text{I}_{g, n+1, \beta}^X (\underline{\gamma} \otimes \id) &= q_{n+1 *}\left(e_{(n+1)}^*(\underline{\gamma}\otimes \id)\cap {[\mathcal{K}_{g,n+1}(\sfrac{\calX}{k}, \beta)]}^\virt \right)\\
&= q_{n+1 *}i_*\left(\hat{e}^*(\underline{\gamma}\otimes \id)\cap i^!{[\mathcal{K}_{g,n+1}(\sfrac{\calX}{k}, \beta)]}^\virt \right)\\
&= \widetilde{q}_*j_*\left(\hat{e}^*(\underline{\gamma}\otimes \id)\cap j^*\widetilde{\phi}^*{[\KgnXbk]}^\virt \right)\\
&= \widetilde{q}_*\left(j_*\hat{e}^*(\underline{\gamma}\otimes \id)\cap \widetilde{\phi}^*{[\KgnXbk]}^\virt \right)\\
&= \widetilde{q}_*\widetilde{\phi}^*\left(e_{(n)}^*(\underline{\gamma})\cap {[\KgnXbk]}^\virt \right)\\
&= \phi_n^*{q}_{n *}\left(e_{(n)}^*(\underline{\gamma})\cap {[\KgnXbk]}^\virt \right)\\
&= \phi_n^* \Ignb(\underline{\gamma})\text{.}
\end{align*}
The remaining part of the proof follows from the same arguments of \cite{AGV08} 8.2.1.
\end{proof}
\subsubsection{Divisor}
We have, for all $\gamma
\in {H^2(X)}$, \[\phi_{*}\text{I}_{g, n+1,
  \beta}^\calX(\bullet \otimes \gamma)=(\beta \cdot
\gamma)\Ignb(\bullet)\text{.}\]
\begin{proof}
Consider the functor \[\overline{\phi} \colon\mathcal{K}_{g,n+1}(\sfrac{X}{k},
\beta)\rightarrow \KgnXbk
\] which forgets the last gerbe and stabilizes, and let
\[\widetilde{\phi}=\phi\times e_{n+1}\colon \mathcal{K}_{g,n+1}(\sfrac{X}{k},
\beta)\rightarrow \KgnXbk\times_k\inertia\].  By the K\"{u}nneth
formula (\cite{milne} VI.8), we can write \[\widetilde{\phi}_*
{[\mathcal{K}_{g,n+1}(\sfrac{X}{k}, \beta)]}^\virt=
{[\KgnXbk]}^\virt\otimes\beta' + \alpha\text{,}\] where $\beta' \in
{H^*(\inertia)}$ and $\alpha\in
H^{m}(\KgnXbk)\otimes H^l(\inertia)$,
with $m$ less than the degree of
${[\KgnXbk]}^\virt$. The class
$\beta'$ can be calculated by restricting to what happens over a
generic point of $\KgnXbk$. Representing such a point by $\xi =(\calC, \Sigma_1
,\ldots,\Sigma_n,f)$, we have the cartesian diagram
  \[
  \begin{tikzpicture}[xscale=3.8,yscale=-2.0]
    \node (A0_0) at (0, 0) {$\calC$};
    \node (A0_1) at (1, 0) {$\xi \times_k \inertia$};
    \node (A0_2) at (2, 0) {$\xi$};
    \node (A1_0) at (0, 1) {$\mathcal{K}_{g,n+1}(\sfrac{X}{k}, \beta)$};
    \node (A1_1) at (1, 1) {$\mathcal{K}_{g,n}(\sfrac{X}{k}, \beta) \times_k \inertia$};
    \node (A1_2) at (2, 1) {$\mathcal{K}_{g,n}(\sfrac{X}{k}, \beta)$};
    \path (A0_0) edge [->]node [auto] {$\scriptstyle{f}$} (A0_1);
    \path (A0_1) edge [->]node [auto] {$\scriptstyle{}$} (A1_1);
    \path (A1_0) edge [->]node [auto] {$\scriptstyle{\widetilde{\phi}}$} (A1_1);
    \path (A0_2) edge [->]node [auto] {$\scriptstyle{i}$} (A1_2);
    \path (A1_1) edge [->]node [auto] {$\scriptstyle{\pi}$} (A1_2);
    \path (A0_0) edge [->]node [auto] {$\scriptstyle{}$} (A1_0);
    \path (A0_1) edge [->]node [auto] {$\scriptstyle{}$} (A0_2);
  \end{tikzpicture}
  \]
where, for $\xi$ generic, the map $i$ is a regular embedding,
hence 
\begin{equation*}
i^!\widetilde{\phi}_{*} {[\mathcal{K}_{g,n+1}(\sfrac{X}{k},
    \beta)]}^\virt=f_*i^!{[\mathcal{K}_{g,n+1}(\sfrac{X}{k},
    \beta)]}^\virt = f_*[\calC]= \beta\text{,}
\end{equation*}
on the other
hand
\begin{equation*}
i^! \widetilde{\phi}_{*} {[\mathcal{K}_{g,n+1}(\sfrac{X}{k},
    \beta)]}^\virt=i^!\left({[\mathcal{K}_{g,n}(\sfrac{X}{k},
    \beta)]}^\virt \otimes \beta' +
\alpha\right)
=\beta'\text{.}
\end{equation*}
 It follows that $\beta'= \beta$. Then
\begin{align*}
  \phi_{*}\text{I}_{g, n+1, \beta}^\calX(\underline{\gamma} \otimes
  \gamma) &= \phi_{*}q_{n+1
    *}\left(e_{(n+1)}^*(\underline{\gamma}\otimes \gamma) \cap
    {[\mathcal{K}_{g,n+1}(\sfrac{\calX}{k}, \beta)]}^\virt\right)\\
  &= q_{n
    *}\pi_*\widetilde{\phi}_*\left(\widetilde{\phi}^*{(e_{(n)}\times
      \id)}^*(\underline{\gamma}\otimes \gamma) \cap
    {[\mathcal{K}_{g,n+1}(\sfrac{\calX}{k}, \beta)]}^\virt\right)\\
  &= q_{n *}\pi_*\left({(e_{(n)}\times
      \id)}^*(\underline{\gamma}\otimes \gamma) \cap
    \widetilde{\phi}_*{[\mathcal{K}_{g,n+1}(\sfrac{\calX}{k}, \beta)]}^\virt\right)\\
  &= q_{n *}\pi_*\left({(e_{(n)}\times
      \id)}^*(\underline{\gamma}\otimes \gamma) \cap
    \left({[\mathcal{K}_{g,n}(\sfrac{\calX}{k}, \beta)]}^\virt\times \beta +\alpha\right)\right)\\
  &= q_{n *}\left(e_{(n)}^*(\underline{\gamma}) \cap
    {[\mathcal{K}_{g,n}(\sfrac{\calX}{k}, \beta)]}^\virt\right)(\beta\cdot\gamma)\\
  &=(\beta \cdot \gamma)\Ignb(\underline{\gamma}) \text{.}\qedhere
\end{align*}
\end{proof}
\subsubsection{Splitting}
Let $g_1,g_2,n_1,n_2\geq 0$ be integers with $2g_j + n_j +1 \geq 3$,
and set $g=g_1+g_2$, $n=n_1 + n_2$. Let \[\gl \colon
\overline{\mathcal{M}}_{\sfrac{g_1, n_1 + 1}{k}} \times_k
\overline{\mathcal{M}}_{\sfrac{g_2, n_2 + 1}{k}} \rightarrow
\mgnk\text{,}\] be the natural functor that identifies the last marked
points. Let $\underline{\gamma} = \gamma_1 \otimes \cdots \otimes
\gamma_n$, then \[\gl^!\Ignb (\underline{\gamma}) = \sum_{\beta_1+ \beta_2=\beta} \text{I}_{g_1, n_1 +1, \beta_1}^\calX\otimes\text{I}_{g_2,
  n_2 +1, \beta_2}^\calX (\underline{\gamma} \otimes
[\Delta])\text{,}\] where
$\Delta$ is the diagonal in $\inertia^2$ (\ref{delta}).
\begin{proof}
  Let us notice that ${A_1(\sfrac{X}{k})}_+$ is a commutative
  semigroup then, by effectivity, the sum is finite. Denote for
  simplicity $\mathcal{K}^{(\beta_j)}=\mathcal{K}_{g_j,n_j+2}(\sfrac{\calX}{k},
  \beta_j)$ for $j=1,2$. Let us consider the following commutative diagram
  \[
  \begin{tikzpicture}[xscale=3.8,yscale=-2.0]
    \node (A0_0) at (0, 0) {$\mathcal{K}^{(\beta_1)}\times_k \mathcal{K}^{(\beta_2)}$};
    \node (A0_1) at (1, 0) {$\mathcal{K}^{(\beta_1)}\times_{\inertia} \mathcal{K}^{(\beta_2)}$};
    \node (A0_2) at (2, 0) {$\KgnXbk$};
    \node (A1_1) at (1, 1) {$\overline{\mathcal{M}}_{\sfrac{g_1, n_1 + 1}{k}} \times_k
\overline{\mathcal{M}}_{\sfrac{g_2, n_2 + 1}{k}} $};
    \node (A1_2) at (2, 1) {$\mgnk$};
    \path (A0_1) edge [->]node [auto] {$\scriptstyle{\widetilde{q}}$} (A1_1);
    \path (A1_1) edge [->]node [auto] {$\scriptstyle{\gl}$} (A1_2);
    \path (A0_2) edge [->]node [auto] {$\scriptstyle{q}$} (A1_2);
    \path (A0_0) edge [->]node [below] {$\scriptstyle{q_{1,2}\phantom{m}}$} (A1_1);
    \path (A0_1) edge [->]node [above] {$\scriptstyle{\widetilde{\Delta}}$} (A0_0);
    \path (A0_1) edge [->]node [auto] {$\scriptstyle{}$} (A0_2);
  \end{tikzpicture}
  \]
where the square is cartesian by Proposition~\ref{prop_gl1}. Moreover, we have the following cartesian diagram
  \[
  \begin{tikzpicture}[xscale=3.4,yscale=-1.8]
    \node (A0_0) at (0, 0) {$\mathcal{K}^{(\beta_1)}\times_{\inertia} \mathcal{K}^{(\beta_2)}$};
    \node (A0_1) at (1, 0) {$\mathcal{K}^{(\beta_1)}\times_k \mathcal{K}^{(\beta_2)}$};
    \node (A1_0) at (0, 1) {$\inertia^{n+1}$};
    \node (A1_1) at (1, 1) {$\inertia^{n+2}$};
    \path (A0_0) edge [->]node [auto] {$\scriptstyle{\widetilde{\Delta}}$} (A0_1);
    \path (A0_0) edge [->]node [left] {$\scriptstyle{\widetilde{e}}$} (A1_0);
    \path (A0_1) edge [->]node [auto] {$\scriptstyle{e_{1,2}}$} (A1_1);
    \path (A1_0) edge [->]node [auto] {$\scriptstyle{\id \times \Delta}$} (A1_1);
  \end{tikzpicture}
  \]
  By
  Proposition~\ref{prop_gl}, \[\gl^!{[\KgnXbk]}^{\virt}=\sum_{\beta_1+ \beta_2=\beta}\Delta^!({[\mathcal{K}_{g_1,n_1+1}(\calX,
    \beta_1)]}^\virt\times {[\mathcal{K}_{g_2,n_2+1}(\calX,
    \beta_2)]}^\virt)\text{.}\] Then we have
\begin{align*}
  \gl^!\Ignb (\underline{\gamma}) &=\gl^!q_*\left(e^*(\underline{\gamma})\cap {[\KgnXbk]}^\virt\right)\\
  &=\widetilde{q}_*\gl^!\left(e^*(\underline{\gamma})\cap {[\KgnXbk]}^\virt\right)\\
  &=\widetilde{q}_*\left(\widetilde{e}^*\pi^*(\underline{\gamma})\cap \gl^!{[\KgnXbk]}^\virt\right)\\
  &=\sum_{\beta_1+ \beta_2=\beta}q_{1,2 *}\widetilde{\Delta}_*\!\left(\widetilde{e}^*(\underline{\gamma}\otimes \id)\cap \Delta^!({[\mathcal{K}_{g_1,n_1+1}(\calX,
    \beta_1)]}^\virt\!\times\!{[\mathcal{K}_{g_2,n_2+1}(\calX,
    \beta_2)]}^\virt)\right)\\
  &=\sum_{\beta_1+ \beta_2=\beta}q_{1,2 *}\left(\widetilde{\Delta}_*\widetilde{e}^*(\underline{\gamma}\otimes \id)\cap ({[\mathcal{K}_{g_1,n_1+1}(\calX,
    \beta_1)]}^\virt\times {[\mathcal{K}_{g_2,n_2+1}(\calX,
    \beta_2)]}^\virt)\right)\\
  &=\sum_{\beta_1+ \beta_2=\beta}q_{1,2 *}\left(e_{1,2}^*(\underline{\gamma}\otimes [\Delta])\cap ({[\mathcal{K}_{g_1,n_1+1}(\calX,
    \beta_1)]}^\virt\times {[\mathcal{K}_{g_2,n_2+1}(\calX,
    \beta_2)]}^\virt)\right)\\
&= \sum_{\beta_1+ \beta_2=\beta} \text{I}_{g_1, n_1 +1, \beta_1}^\calX\otimes\text{I}_{g_2,
  n_2 +1, \beta_2}^\calX (\underline{\gamma} \otimes
[\Delta])\text{.}\qedhere
\end{align*}
\end{proof}
\subsubsection{Genus reduction}
Let $\gl \colon \overline{\mathcal{M}}_{\sfrac{g-1, n + 2}{k}}
\rightarrow \mgnk$ be the natural functor that identifies the last
gerbes.  We have \[\gl^!\Ignb(\bullet) =
\text{I}_{g-1, n+2, \beta}^\calX(\bullet \otimes
[\Delta])\text{,}\] where $\Delta$ is the diagonal in $\inertia^2$ (\ref{delta}).
\begin{proof}
Let us consider the following commutative diagram
  \[
  \begin{tikzpicture}[xscale=4.8,yscale=-2.4]
    \node (A0_0) at (0, 0) {$\mathcal{K}_{g-1,n+2}(\sfrac{\calX}{k}, \beta)$};
    \node (A0_1) at (1.1, 0) {$\mathcal{K}_{g-1,n+2}(\sfrac{\calX}{k}, \beta)\times_{\inertia^2} \inertia$};
    \node (A0_2) at (2, 0) {$\KgnXbk$};
    \node (A1_1) at (1.1, 1) {$\overline{\mathcal{M}}_{\sfrac{g-1, n + 2}{k}}$};
    \node (A1_2) at (2, 1) {$\mgnk$};
    \path (A0_1) edge [->]node [auto] {$\scriptstyle{\widetilde{q}}$} (A1_1);
    \path (A1_1) edge [->]node [auto] {$\scriptstyle{\gl}$} (A1_2);
    \path (A0_2) edge [->]node [auto] {$\scriptstyle{q_n}$} (A1_2);
    \path (A0_0) edge [->]node [below] {$\scriptstyle{q_{n+2}\phantom{m}}$} (A1_1);
    \path (A0_1) edge [->]node [above] {$\scriptstyle{\widetilde{\Delta}}$} (A0_0);
    \path (A0_1) edge [->]node [auto] {$\scriptstyle{}$} (A0_2);
  \end{tikzpicture}
  \]
where the square is cartesian. Moreover, we have the following cartesian diagram
  \[
  \begin{tikzpicture}[xscale=4.8,yscale=-2.4]
    \node (A0_0) at (0, 0) {$\mathcal{K}_{g-1,n+2}(\sfrac{\calX}{k}, \beta)\times_{\inertia^2} \inertia$};
    \node (A0_1) at (1, 0) {$\mathcal{K}_{g-1,n+2}(\sfrac{\calX}{k}, \beta)$};
    \node (A1_0) at (0, 1) {$\inertia^{n+1}$};
    \node (A1_1) at (1, 1) {$\inertia^{n+2}$};
    \path (A0_0) edge [->]node [auto] {$\scriptstyle{\widetilde{\Delta}}$} (A0_1);
    \path (A0_0) edge [->]node [left] {$\scriptstyle{\widetilde{e}}$} (A1_0);
    \path (A0_1) edge [->]node [auto] {$\scriptstyle{e_{(n+2)}}$} (A1_1);
    \path (A1_0) edge [->]node [auto] {$\scriptstyle{\id \times \Delta}$} (A1_1);
  \end{tikzpicture}
  \]
  By
  Proposition~\ref{prop_gl}, \[\gl^!{[\KgnXbk]}^{\virt}=\Delta^!({[\mathcal{K}_{g-1,n+2}(\calX,\beta)]}^\virt)\text{.}\] Then we have
\begin{align*}
\gl^!\Ignb(\underline{\gamma}) &= \gl^!q_{n *}\left(e_{(n)}^*(\underline{\gamma})\cap {[\KgnXbk]}^\virt\right)\\
&= \widetilde{q}_*\gl^!\left(e_{(n)}^*(\underline{\gamma})\cap {[\KgnXbk]}^\virt\right)\\
&= \widetilde{q}_*\left(\widetilde{e}^*\pi^*(\underline{\gamma})\cap \gl^!{[\KgnXbk]}^\virt\right)\\
&= q_{n+2 *}\widetilde{\Delta}_*\left(\widetilde{e}^*(\underline{\gamma}\otimes \id)\cap \Delta^!{[\mathcal{K}_{g-1,n+2}(\calX,\beta)]}^\virt\right)\\
&= q_{n+2 *}\left(\widetilde{\Delta}_*\widetilde{e}^*(\underline{\gamma}\otimes \id)\cap {[\mathcal{K}_{g-1,n+2}(\calX,\beta)]}^\virt\right)\\
&= q_{n+2 *}\left(e_{(n+2)}^*(\underline{\gamma}\otimes [\Delta])\cap {[\mathcal{K}_{g-1,n+2}(\calX,\beta)]}^\virt\right)\\
&=
\text{I}_{g-1, n+2, \beta}^\calX(\underline{\gamma} \otimes
[\Delta])\text{.}\qedhere
\end{align*}
\end{proof}
\section{Genus zero invariants in positive characteristic}\label{section_genus0}
\subsection{Gromov-Witten potential}
Let $\calX$ be a smooth proper tame Deligne-Mumford stack of finite
presentation over an algebraically closed field $k$ (of arbitrary characteristic), admitting a
projective coarse moduli scheme $X$.  Fix
$\beta \in A_1(\sfrac{X}{k})$ and $n \geq 0$.
Let $l$ be a prime different from the characteristic of $k$.
\begin{remark}
  Recall that we defined on the group $H_{\st}^*(\calX)=H^*(\inertia)$
  the following grading
  $H_\st^m(\calX)=\bigoplus_{\Omega}H^{m-2\age(\Omega)}(\Omega)$,
  where the sum is taken over all connected components $\Omega$ of
  $\inertia$.  By \cite{milne} V.1.11, $H_\st^*(\calX)= \sum_r
  H^r(\inertia, \bbQ_l(\overline{r}))$ is finitely generated over
  $\bbQ_l$. Let $T_0=1, T_1, \ldots, T_m$ be generators for
  $H_\st^*(\calX)$. For each $i=1, \ldots, m$, we introduce a variable
  $t_i$ of the same degree of $T_i$, such that the $t_i$ supercommute,
  which means \[t_it_j= {(-1)}^{\deg t_i \deg t_j}t_jt_i\text{,}\] and
  $t_i^2=0$ if $t_i$ has odd degree.
\end{remark}
\begin{remark}
  If $\gamma_i\in H_\st^{m_i}(\calX)$ then $\langle \Inb \rangle
  (\gamma_1\otimes \cdots \otimes \gamma_n) \in \bbQ_l$ is zero unless
  \[\sum_{i+1}^nm_i=2(\dim_k\calX+c_1(T_{\sfrac{\calX}{k}})\cdot \beta)\text{.}\]
\end{remark}
\begin{notat}
We denote the vector $(a_0, \ldots, a_m)$ as $\ua$; we set $|\ua|= a_0
+ \cdots + a_m$ and $\ua!=a_0! \cdots a_m!$. Moreover we set $\langle
\Inb\rangle=0$ for $n<3$.
\end{notat}
\begin{definit}
Let $\gamma = \sum_{i=0}^mt_i T_i$ (regarding $T_i$ and $t_i$ as
supercommuting variables). We define the \emph{genus $0$ Gromov-Witten
  potential} as the formal series \[\Phi (\gamma)=\sum_{n\geq
  0}\sum_{\beta \in A_1(\sfrac{X}{k})}\frac{1}{n!}\langle\Inb\rangle(\gamma^n) q^\beta\text{,}\]
where $q^\beta$ is a free variable of degree $\beta \cdot c_1(T_{\sfrac{\calX}{k}})$ and \[\frac{1}{n!}\langle\Inb\rangle(\gamma^n)=
\sum_{|\ua|=n}\epsilon(\ua)\langle\Inb\rangle(T^{\ua})\frac{t^{\ua}}{\ua!}\text{,}\]
with $\epsilon (\ua)=\pm 1$ determined by \[{(t_0T_0)}^{a_0} \cdots
{(t_mT_m)}^{a_m}= \epsilon (\ua)T_0^{a_0}\cdots
T_m^{a_m}t_0^{a_0}\cdots t_m^{a_m}\text{.}\]
\end{definit}
\begin{remark}
By effectivity axiom, the Gromov-Witten potential is a formal series
in $\calR=R\llbracket t_0, \ldots, t_m\rrbracket$, with $R = \bbQ_l \llbracket
q^\beta;\, \beta \in {A_1(\sfrac{X}{k})}_+\rrbracket$.
\end{remark}
\subsection{Quantum product}
By \cite{milne} VI.8, $H^*(\inertia \times_k \inertia)=
H^*(\inertia)\otimes H^*(\inertia)$. Let $\Delta$ be the diagonal in
$\inertia^2$ (\ref{delta}), then \[[\Delta]= \sum_{e,f}g^{ef}T_e\otimes
T_f\text{.}\]
\begin{definit}
We define \[T_i * T_j = \sum_{e,f}\frac{\partial^3\Phi}{\partial t_i
  \partial t_j \partial t_e}g^{ef}T_f\text{.}\] Extending this
linearly gives the \emph{(big) quantum product} on $H_\st^*(\calX, \calR)$.
\end{definit}
\begin{remark}
Notice that the Gromov-Witten invariants with $n<3$ do not affect the
quantum product.
\end{remark}
\begin{lemma}\label{lemma_d3}
For all $i,j,h$, we have \[\frac{\partial^3\Phi(\gamma)}{\partial t_i \partial
  t_j \partial t_h}= \sum_{n\geq 0}\sum_{\beta\in
  A_1(\sfrac{X}{k})}\frac{1}{n!}\langle\I_{0,n+3,\beta}^\calX\rangle(T_i\otimes
T_j\otimes T_h \otimes \gamma^n) q^\beta\text{.}\]
\end{lemma}
\begin{proof}
For simplicity, we will assume that $H_\st^*(\calX, \mathcal{R})$ has only
even cohomology so that we don't have to worry about signs. We have
\[\frac{\partial^3\Phi(\gamma)}{\partial t_i \partial t_j \partial
    t_h} =\frac{\partial^3}{\partial t_i \partial t_j \partial
    t_h}\sum_{n\text{, }\beta}
  \sum_{|\ua|=n}\langle\Inb\rangle(T^{\ua})\frac{t^{\ua}}{\ua!}q^\beta=\sum_{n\text{, }\beta}
  \sum_{|\ua|=n}\langle\Inb\rangle(T^{\ua'})\frac{t^{\ua'}}{\ua'!}q^\beta\text{,}\]
where $\ua'=\ua -e_i-e_j-e_h$ and $e_i=(0,\ldots,0,1,0,\ldots,0)$ with
$1$ in the $i$-th position. Moreover
\begin{align*}
  \sum_{n\text{,
    }\beta}\frac{1}{n!}\langle\I_{0,n+3,\beta}^\calX\rangle(T_i\otimes
  T_j\otimes T_h \otimes \gamma^n) q^\beta &=\sum_{n\text{, }\beta}
  \sum_{|\ua|=n}\langle\I_{0,n+3,\beta}^\calX\rangle(T_i\otimes T_j\otimes
  T_h \otimes T^{\ua})\frac{t^{\ua}}{\ua!}q^\beta\\ &=\sum_{n\text{,
    }\beta}
  \sum_{|\ua|=n+3}\langle\I_{0,n+3,\beta}^\calX\rangle(T^{\ua+e_i+e_j+e_h})\frac{t^{\ua}}{\ua!}q^\beta\text{.}\qedhere
\end{align*}
\end{proof}
\begin{theorem}[WDVV equation]\label{theorem_wdvv}
The Gromov-Witten potential satisfies the
equation \[\sum_{e,f}\frac{\partial^3\Phi}{\partial t_i \partial t_j
  \partial t_e}g^{ef}\frac{\partial^3\Phi}{\partial t_f \partial t_h
  \partial t_l}={(-1)}^{\deg t_i (\deg t_j + \deg
  t_h)}\sum_{e,f}\frac{\partial^3\Phi}{\partial t_j \partial t_h
  \partial t_e}g^{ef}\frac{\partial^3\Phi}{\partial t_f \partial t_i
  \partial t_l}\text{,}\] for all $i,j,h,l$.
\end{theorem}
\begin{proof}
For simplicity, we will assume that $H_\st^*(\calX, \mathcal{R})$ has only
even cohomology so that we don't have to worry about signs. 
If we
set \[F(ij|hl)=\frac{\partial^3\Phi}{\partial t_i \partial t_j
  \partial t_e}g^{ef}\frac{\partial^3\Phi}{\partial t_f \partial t_h
  \partial t_l}\text{,}\] then we want to show that
$F(ij|hl)= F(jh|il)$. Consider the following cartesian diagram
  \[
  \begin{tikzpicture}[xscale=3.8,yscale=-2.0]
    \node (A1_0) at (-0.5, 1) {$D(ij|hl)$};
    \node (A1_1) at (0.5, 1) {$\overline{\mathcal{M}}_{\sfrac{0,n+4}{k}}$};
    \node (A2_0) at (-0.5, 2) {$\overline{\mathcal{M}}_{\sfrac{0,\{i,j\}\cup \bullet}{k}}\times_k \overline{\mathcal{M}}_{\sfrac{0,\{h,l\}\cup \bullet}{k}}$};
    \node (A2_2) at (-1.4, 2) {$\spec k$};
    \node (A2_1) at (0.5, 2) {$\overline{\mathcal{M}}_{\sfrac{0,4}{k}}$};
    \path (A2_2) edge [-,double distance=1.5pt]node [auto] {$\scriptstyle{}$} (A2_0);
    \path (A2_0) edge [->]node [auto] {$\scriptstyle{\gl}$} (A2_1);
    \path (A1_0) edge [->]node [auto] {$\scriptstyle{}$} (A1_1);
    \path (A1_0) edge [->]node [auto] {$\scriptstyle{}$} (A2_0);
    \path (A1_1) edge [->]node [auto] {$\scriptstyle{\rho}$} (A2_1);
  \end{tikzpicture}
  \]
where the image of $\gl$ is a boundary point of
$\overline{\mathcal{M}}_{\sfrac{0,4}{k}}\cong \bbP_k^1$. Since the boundary
points are linearly equivalent, the same is true for the fibers of
$\rho$ over these points, hence $D(ij|hl)$ and $D(jh|il)$ are linearly
equivalent divisors in $\overline{\mathcal{M}}_{\sfrac{0,n+4}{k}}$. Let $A \cup
B$ be a partition of $\{1, \ldots, n+4\}$ such that $i,j\in A$ and
$h,l \in B$. Let us set $\overline{\mathcal{M}}_{A,B}=\overline{\mathcal{M}}_{\sfrac{0,A \cup \bullet}{k}}\times_k\overline{\mathcal{M}}_{\sfrac{0,B \cup \bullet}{k}}$ and
form the fiber square
  \[
  \begin{tikzpicture}[xscale=2.4,yscale=-1.6]
    \node (A0_0) at (0, 0) {$D(A|B)$};
    \node (A0_1) at (1, 0) {$D(ij|hl)$};
    \node (A1_0) at (0, 1) {$\overline{\mathcal{M}}_{A,B}$};
    \node (A1_1) at (1, 1) {$\overline{\mathcal{M}}_{\sfrac{0,n+4}{k}}$};
    \path (A0_0) edge [->]node [auto] {$\scriptstyle{}$} (A0_1);
    \path (A1_0) edge [->]node [auto] {$\scriptstyle{\gl}$} (A1_1);
    \path (A0_1) edge [->]node [auto] {$\scriptstyle{}$} (A1_1);
    \path (A0_0) edge [->]node [auto] {$\scriptstyle{}$} (A1_0);
  \end{tikzpicture}
  \]
  then $D(ij|hl)= \bigsqcup_{\substack{A\cup B = \{1, \ldots, n+4\}\\
      i,j \in A\text{, }h,l\in B}}D(A|B)$. We set
  \[\overline{\mathcal{K}}^{(\beta_1,\beta_2)}=\mathcal{K}_{0,A\cup
    \bullet}(\sfrac{\calX}{k}, \beta_1)\times_k\mathcal{K}_{0,B\cup
    \bullet}(\sfrac{\calX}{k}, \beta_2)\text{.}\] Let us set for
  simplicity $\gamma_{n_1}= T_i\otimes T_j\otimes \gamma^{n_1}$ and
  $\gamma_{n_2}= T_h\otimes T_l\otimes \gamma^{n_2}$. Then, by
  Lemma~\ref{lemma_d3} and splitting axiom,
\begin{align*}
F(ij|hl)&= \sum_{\beta_1, \beta_2,n_1,
  n_2,e,f}\frac{1}{n_1!n_2!}\langle\I_{0,n_1+3,\beta_1}^\calX\rangle(T_e\otimes
\gamma_{n_1})g^{ef}\langle\I_{0,n_2+3,\beta_2}^\calX\rangle(T_f\otimes \gamma_{n_2}) q^{\beta_1+\beta_2}\\ &= \sum_{\beta,
  n}\sum_{\substack{\beta_1+ \beta_2=\beta\\ n_1
    +n_2=n}}\sum_{e,f}\frac{1}{n_1!n_2!}\int_{{[\mathcal{K}^{(\beta_1,
        \beta_2)}]}^\virt}g^{e,f}e_{(1,2)}^*(T_e \otimes \gamma_{n_1}\otimes T_f\otimes
\gamma_{n_2}) q^{\beta}\\ &= \sum_{\beta, n}\sum_{\substack{\beta_1+
    \beta_2=\beta\\ n_1
    +n_2=n}}\frac{1}{n_1!n_2!}\int_{{[\mathcal{K}^{(\beta_1,
        \beta_2)}]}^\virt}e_{(1,2)}^*(\gamma_{n_1}\otimes [\Delta]\otimes \gamma_{n_2}) q^{\beta}\\
&=
\sum_{\beta, n}\sum_{\substack{A \cup B =
    \{1, \ldots, n+4\}\\ i,j \in A \text{, }h,l \in
    B}}\frac{1}{n!}\int_{\overline{\mathcal{M}}_{A,B}}\gl^!\I_{0,n+4,\beta}^\calX(T_i\otimes
T_j\otimes T_h\otimes T_l\otimes \gamma^{n}) q^{\beta}\\
&=
\sum_{\beta, n}\sum_{\substack{A \cup B =
    \{1, \ldots, n+4\}\\ i,j \in A \text{, }h,l \in
    B}}\frac{1}{n!}\int_{D(A|B)}\I_{0,n+4,\beta}^\calX(T_i\otimes
T_j\otimes T_h\otimes T_l\otimes \gamma^{n}) q^{\beta}\\
&=
\sum_{\beta, n}\frac{1}{n!}\int_{D(ij|hl)}\I_{0,n+4,\beta}^\calX(T_i\otimes
T_j\otimes T_h\otimes T_l\otimes \gamma^{n}) q^{\beta}\text{.}
\end{align*}
Since $D(ij|hl)$ and $D(jh|il)$ are linerly equivalent, it follows that $F(ij|hl)=F(jh|il)$.
\end{proof}
\begin{prop}\label{prop_*}
The quantum product is supercommutative with identity $T_0$ and associative.
\end{prop}
\begin{proof}
By Lemma~\ref{lemma_d3} and $S_n$-covariance axiom,
\begin{align*}
T_i * T_j &= \sum_{\beta, n,
  e,f}\frac{1}{n!}\langle\I_{0,n+3,\beta}^\calX\rangle(T_i\otimes
T_j\otimes T_e \otimes \gamma^n) g^{ef}T_fq^\beta\\
 &= \sum_{\beta, n,
  e,f}\frac{1}{n!}{(-1)}^{\deg T_i \deg T_j}\langle\I_{0,n+3,\beta}^\calX\rangle(T_j\otimes
T_i\otimes T_e \otimes \gamma^n) g^{ef}T_fq^\beta\\
&= {(-1)}^{\deg T_i \deg T_j}T_j * T_i\text{.}
\end{align*}
Let $\Delta \colon \inertia \rightarrow \inertia^2$ be the diagonal (\ref{delta})
and let $p_i\colon \inertia^2\rightarrow \inertia$ be the natural
projections for $i=1,2$. By the fundamental class axiom,
\begin{align*}
T_i &= p_{2 *}\Delta_*(\Delta^!p_1^*\iota^*(T_i))\\
&= p_{2 *}(p_1^*\iota^*(T_i)\cup [\Delta])\\
&= \sum_{e,f}g^{ef}p_{2 *}((\iota^*(T_i)\otimes T_0)\cup (T_e\otimes T_f))\\
&= \sum_{e,f}g^{ef}p_{2 *}((\iota^*(T_i)\cup T_e)\otimes T_f)\\
&= \sum_{e,f}\langle \I_{0,3,0}^\calX \rangle (T_0
\otimes T_i \otimes T_e)g^{ef}T_f\text{.}
\end{align*}
Moreover, we have $\langle\I_{0,n+3,\beta}^\calX\rangle(\bullet \otimes
T_0)=0$ unless $\beta=0$ and $n=3$. Therefore
\begin{align*}
T_0 * T_i &= \sum_{\beta, n,
  e,f}\frac{1}{n!}\langle\I_{0,n+3,\beta}^X\rangle(T_0\otimes
T_i\otimes T_e \otimes \gamma^n) g^{ef}T_fq^\beta\\
 &= \sum_{e,f}\langle\I_{0,3,0}^X\rangle(T_0\otimes
T_i\otimes T_e) g^{ef}T_f= T_i\text{.}
\end{align*}
Finally, we prove that the quantum product is associative.  For
simplicity, we will assume that $H_\st^*(\calX, \mathcal{R})$ has only even
cohomology so that we don't have to worry about signs. We have \[(T_i
* T_j) * T_h = \sum_{e,f}\frac{\partial^3\Phi}{\partial t_i \partial
  t_j \partial t_e}g^{ef}T_e * T_h=
\sum_{c,d,e,f}\frac{\partial^3\Phi}{\partial t_i \partial t_j \partial
  t_e}g^{ef}\frac{\partial^3\Phi}{\partial t_f \partial t_h \partial
  t_c}g^{cd}T_d\] and \[T_i * (T_j * T_h) = {(-1)}^{\deg T_i (\deg T_j
  + \deg T_h)}(T_j * T_h)*T_i\text{,}\] since the quantum product is
supercommutative. Therefore, associativity follows from
Theorem~\ref{theorem_wdvv}.
\end{proof}
\subsection{Reconstruction for genus zero Gromov-Witten invariants.}
\begin{theorem}\label{theorem_reconstr}
  If $H_\st^*(\calX)$ is generated by $H_\st^2(\calX)$ then every
  genus zero Gromov-Witten invariant can be uniquely reconstructed
  starting with the following system of Gromov-Witten
  invariants \[\insieme{\I_{0,3,\beta}^\calX(\gamma_1 \otimes
    \gamma_2\otimes \gamma_3)}{\beta \cdot c_1(T_{\sfrac{\calX}{k}})\leq
    \dim_k \calX +1\text{, }\deg \gamma_3=2}\text{.}\]
\end{theorem}
\begin{proof}
Apply the WDVV equation
(Theorem~\ref{theorem_wdvv}) to $\gamma_1\otimes \cdots \otimes \gamma_{n+1}$
with indeces $\{i,j,h,l\}=\{1,2,n,n+1\}$. Let us define a partial
order on pairs $(\beta, n)$, with $n \geq 3$ and $\beta\in {A_1(\sfrac{X}{k})}_+$, by setting $(\beta, n)>(\beta', n')$ if and only if either $\beta=
\beta' + \beta''$ or $\beta= \beta'$ and $n > n'$. Then there are four
terms of higher order in the WDVV equation each of the
form \[\I_{a,b}=\sum_{e,f}\langle\I_{0,3,0}^\calX\rangle(\gamma_a\otimes
\gamma_b\otimes T_e)g^{ef}\langle \I_{0,n-1,\beta}^\calX\rangle(T_f
\otimes (\otimes_{s \neq a,b}\gamma_s))\text{,}\] with $(a,b) \in
\{(1,2),(n,n+1),(2,n),(1,n+1)\}$. As shown in the proof of
Proposition~\ref{prop_*}, we have
\[\gamma_a \cup \gamma_b =\sum_{e,f}\langle \I_{0,3,0}^\calX \rangle (\gamma_a \otimes \gamma_b
\otimes T_e)g^{ef}T_f\text{,}\] hence $\I_{a,b}= \langle
\I_{0,n-1,\beta}^\calX\rangle (\gamma_a\cup \gamma_b \otimes (\otimes_{s
  \neq a,b}\gamma_s))$. Let us consider
$\langle\I_{0,n,\beta}^\calX\rangle(\gamma_1 \otimes \cdots \otimes
\gamma_n)$. If $\deg \gamma_n=2$, then we can apply divisor axiom to
reduce $n$. Otherwise, since $H_\st^*(\calX)$ is generated by $H_\st^2(\calX)$, we can
write $\gamma_n=\sum_i \delta_i'\cup \delta_i$, with $\deg
\delta_i=2$. By linearity, we can assume $\gamma_n=\delta' \cup
\delta$, with $\deg \delta=2$. Apply the construction above with
$\gamma_n=\delta'$ and $\gamma_{n+1}=\delta$. Then, by WDVV equation,
we get
\begin{align*}
\pm \langle \I_{0,n-1,\beta}^\calX\rangle (\gamma_1\cup \gamma_2 \otimes
\gamma_3 &\otimes \cdots \gamma_{n-1}\otimes \delta' \otimes \delta)\pm
\langle \I_{0,n-1,\beta}^\calX\rangle (\gamma_1 \otimes \cdots \otimes
\gamma_{n-1}\otimes \delta' \cup\delta)\\ \pm \langle
\I_{0,n-1,\beta}^\calX\rangle (\gamma_1\cup \delta \otimes \gamma_2
\otimes &\cdots \gamma_{n-1}\otimes \delta')\pm \langle
\I_{0,n-1,\beta}^\calX\rangle (\gamma_1\otimes \gamma_2\cup \delta'
\otimes \gamma_ \otimes \cdots \gamma_{n-1}\otimes \delta)=
\\ &=\text{a combination of higher order terms.}
\end{align*}
By divisor axiom, the first and the fourth summands are lifted from
$\overline{\mathcal{M}}_{\sfrac{0,n-1}{k}}$. Moreover in the third summand we
have $\deg \delta' < \deg \gamma_n$. If $\deg \delta'=2$ then, by
divisor axiom, we can reduce $n$, otherwise we repeat this trick and
in a finite number of iterations we will reduce $n$.
Finally, we can apply the procedure described above to
$\langle\I_{0,3,\beta}^\calX\rangle(\gamma_1\otimes\gamma_2\otimes\gamma_3)$
and decrease $\deg \gamma_3\geq 2$.
\end{proof}

{\it E-mail address:} {\tt flavia.poma@gmail.com}

\end{document}